\documentclass[sn-mathphys,Numbered]{sn-jnl}

\usepackage{graphicx}%
\usepackage{multirow}%
\usepackage{amsmath,amssymb,amsfonts}%
\usepackage{amsthm}%
\usepackage{mathrsfs}%
\usepackage[title]{appendix}%
\usepackage{xcolor}%
\usepackage{textcomp}%
\usepackage{manyfoot}%
\usepackage{booktabs}%
\usepackage{algorithm}%
\usepackage{algorithmicx}%
\usepackage{algpseudocode}%
\usepackage{listings}%

\usepackage{amsmath,amsfonts,amstext,amscd,amssymb,a4,enumerate,epsfig,psfrag}
\usepackage[latin1]{inputenc}
\usepackage{graphicx}
\usepackage{amsopn}
\usepackage{color}

\theoremstyle{thmstyleone}%

\theoremstyle{thmstyletwo}%

\theoremstyle{thmstylethree}%

\newcommand{\cO}{{\mathcal{O}}}
\newcommand{\Cc}{{\cal{C}}}

\newcommand{\cB}{{\mathcal{B}}}
\newcommand{\cC}{{\mathcal{C}}}

\newcommand{\cF}{{\mathcal{F}}}

\newcommand{\cS}{{\mathcal{S}}}

\newcommand{\ZZ}{{\mathbb{Z}}}

\newcommand{\RR}{{\mathbb{R}}}
\newcommand{\CC}{{\mathbb{C}}}

\newtheorem{theo} {Theorem}
\newtheorem{lemma} {Lemma}[section]
\newtheorem{prop}[lemma] {Proposition}
\newtheorem{coro}[lemma] {Corollary}

\theoremstyle{remark}

\AtEndEnvironment{defin}{\null\hfill$\diamond$}%
\AtEndEnvironment{example}{\null\hfill$\diamond$}%
\AtEndEnvironment{remark}{\null\hfill$\diamond$}%

\newcommand{\ind}{\operatorname{ind}}

\newcommand{\spam}{\operatorname{span}}
\newcommand{\Ran}{\operatorname{Ran}}
\newcommand{\Ker}{\operatorname{Ker}}

\newcommand{\sign}{{\operatorname{sign}}}

\begin{document}

\title[Multiple solutions by inducing bifurcations]{Using the critical set to induce bifurcations}

\author[1]{\fnm{O.} \sur{Kaminski}}\email{otaviokaminski@gmail.com}
\equalcont{These authors contributed equally to this work.}

\author[2]{\fnm{D. S.} \sur{Monteiro}}\email{diego\_smonteiro@hotmail.com}
\equalcont{These authors contributed equally to this work.}

\author*[3]{\fnm{C.} \sur{Tomei}}\email{carlos.tomei@gmail.com}

\affil[1]{\orgdiv{Departamento de Educação}, \orgname{IBC},  \orgaddress{\street{Av.Pasteur 350} \\ \city{Rio de Janeiro}, \state{RJ}, \country{Brazil}}}

\affil[2]{\orgdiv{Colégio de Aplicação}, \orgname{UERJ},  \orgaddress{\street{R. B. Itapagipe 96}\\ \city{Rio de Janeiro} \state{RJ}, \country{Brazil}}}

\affil*[3]{\orgdiv{Departmento de Matemática}, \orgname{PUC-RIO}, \orgaddress{\street{R. Mq. S. Vicente 225} \\ \city{Rio de Janeiro}, \state{RJ}, \country{Brazil}}}

\abstract{For a function $F: X \to Y$ between real Banach spaces, we show how continuation methods to solve $F(u) = g$  may improve from basic understanding of the critical set $\cC$ of $F$. The algorithm aims at special points with a large number of preimages,  which in turn may be used as initial conditions for standard continuation methods applied to the solution of the desired equation. A geometric model based on the sets $\cC$  and $F^{-1}(F(\cC))$ substantiate our choice of curves $c \in X$ with abundant intersections with $\cC$. 
	
	We consider three classes of examples. First we handle functions  $F: \RR^2 \to \RR^2$, for which the reasoning behind the techniques is visualizable. The second set of examples, between spaces of dimension 15, is obtained by discretizing a nonlinear Sturm-Liouville problem for which special points admit a high number of solutions. Finally, we handle a semilinear elliptic operator, by computing the six solutions   of  an equation of the form $-\Delta - f(u) = g$ studied by Solimini. }



\maketitle

\section{Introduction}\label{introduction}
We consider the classical problem of computing solutions of $F(u)=g$, for a map $F: D \subset X \to Y$ between real Banach spaces. We present a context in which the geometry of the function $F$ can be exploited by {\it inducing bifurcations}. From knowledge of the critical set $\cC$ of  $F$, we indicate curves  $c \subset D$ with substantial intersection with $\cC$, taken as starting points of continuation algorithms.

The algorithm searches for points $g_\ast$ with a large number of preimages, i.e., right hand sides for which $F(u) = g_\ast$ has many solutions. 
Standard continuation methods may then solve $F(u) = g$ for a general  $g$ by inverting a curve in the image joining $g_\ast$ to $g$ starting at a preimage $u_\ast$ of $g_\ast$. (\cite{ALLGOWER,RHEINBOLDT}).

A standard method to obtain preimages of $g$ first extends $F$ to a function $\tilde F(u,t)$ for which a simple curve $d = (u(t),t)$ in the domain usually  satisfies $F(u(0),0) = g$.
Assuming some differentiability, one solves for additional preimages by first identifying {\it bifurcation points} $u(t_c)$, in which $D\tilde{F}(u(t_c), t_c)$ is not surjective for $t_c \in [0,1]$.
Such points yield new branches of solutions $(\tilde u(t),t)$: if they extend to $t=1$, new preimages of $g$ are obtained. Additional bifurcations may arise along such branches. Usually, the specification of the curve $d$ is very restricted and does not allow for a choice of $g = g_\ast$ as above. Continuation methods received a recent boost from ideas by Farrell, Birkisson and Funke \cite{FARRELL}, which improved  an elegant deflation strategy originally suggested by Brown and Gearhart \cite{BROWN}.  

\medskip
The main idea is simple. Clearly, not every curve $c \subset D$ intersecting abundantly the critical set $\Cc$ leads to points with many preimages. By considering very loose geometric structure, we obtain indicators leading to more appropriate curves. The algorithm requires that we may decide  if $u \in D$ is a critical point (i.e., an element of $\cC$) or if  $u \in\partial D$. Inversion near critical points relies on spectral information, in the spirit of Section 3.3 of \cite{UECKER}. 

We explore a {\it geometric model} for  functions $F: D \subset X \to Y$ between spaces of the same dimension\footnote{In infinite dimensions, Jacobians should be Fredholm operators of index zero.}. It combines standard results in analysis and topology outlined in Theorems \ref{covers} and \ref{adjacenttiles} of Section  \ref{model}.  The model fits functions satisfying a weakened form of properness, in particular a class of semilinear differential operators.

\bigskip
\noindent {\bf Geometric model:} Domain and counterdomain split in {\it tiles}, the connected components of $F^{-1} (F(\cC \cup \partial D)) \subset X$ and $F(\cC \cup \partial D) \subset Y$. The restriction of $F$ to a domain tile is a covering map onto an image tile. In particular, the number of preimages of points in an image tile is constant. Between adjacent image tiles, this number changes in a simple fashion described in Theorem \ref{adjacenttiles}.
Adjacent domain tiles separated by an arc of $\cC$ are sent to the same image tile.

\bigskip

This approach started with the study of  proper functions $F: \RR^2 \to \RR^2 $  by  Malta, Saldanha and Tomei in the late eighties (\cite{MST1}), leading to $2\times 2$, a software that computes preimages of $F$, together with other relevant geometrical objects\footnote{Available at http://www.imuff.mat.br/puc-rio/2x2, but supported  by Windows 7.}. Under generic conditions, the authors obtain a characterization of critical sets: given finite sets of curves $\{\cC_i\}$ and $\{\cS_i\}$, and finite points $\{p_{ik} \in \cC_i\}$, one can decide if there is a proper, generic function $F$ whose critical set consists of the curves $C_i$, with images $ \cS_i= F(\cC_i) $ and cusps\footnote{Informally, cusps are the second most frequent critical points, after folds \cite{Whitney}.} at the points $\{p_{ik}\}$. Given a function $F$, $2 \times 2$ obtains some critical curves $\mathcal{C}_i$, its images $\cS_i$ and their cusps $\{p_{ik}\}$: if the characterization does not hold, it provides the program with information about where to search for additional critical curves.

Two important features of the 2-dimensional context do not extend:
(a) a description of the critical set as a list of critical points,
(b) the identification of higher order singularities.
A realistic implementation for $n > 2$ led us to the current text.
Some mathematical concepts in $2 \times 2$ found application in different scenarios (for ODEs,
\cite{TOMEIEBUENO,BURGHELEA,TELES}; for PDEs, \cite{CALNETO,KAMINSKI}).

\bigskip In all examples in this text, the starting point of our arguments lie in the identification of the critical set of the underlying function.  In Sections \ref{model} and \ref{Visual} we present the geometric context  and apply  the algorithm to some visualizable examples, so that the counterparts of bifurcation diagrams can be presented concretely. 

The second class of examples, discussed in Section \ref{SturmLiouville},  arises from the discretization of a nonlinear Sturm-Liouville problem,
$$ F(u) =  - u'' + f(u) = g  , \quad u(0) = 0 = u(\pi)  .$$
For a uniform mesh with spacing $h$, the discretized equation $F_h(u) = g$ has an unexpectedly high number of solutions (\cite{TELES}), and most do not admit a continuous limit, but we consider the problem as an example of interest by itself.  The connected components $\cC_i, i=1, \ldots, n$, of the critical set of $F_h$ are graphs of functions $\gamma_i: V^\perp \to V$, where $V$ is the one dimensional space generated by the ground state of the discretization of $u \mapsto -u''$, which, as is well known, is the evaluation $u_h$ of  $u(x) = \sin x$ at points of the uniform mesh on $[0, \pi]$.  Thus, straight lines in the domain which are parallel to the vector $u_h$ contain many critical points. Parallel lines give rise to additional solutions, for geometric reasons we make clear.

Our research was  inspired by the celebrated Ambrosetti-Prodi theorem (\cite{AP,MM}). Subsequent articles (\cite{AP,BERGER,CALNETO,SMILEY,KAMINSKI}) provided  information about the geometry of semilinear elliptic operators, with  implications to the underlying numerics. In \cite{ALLGOWER2}, Allgower, Cruceanu and Tavener considered numerical solvability of semilinear elliptic equations by first obtaining good approximations of solutions from discretized versions of the problem. A filtering strategy eliminates some discrete solutions which do not yield a continuous limit, and algorithms, somehow tailored to the form of the equations, are presented and exemplified. As they remark, results on the number of solutions are abundant, but not really precise. 
 
The third context, in Section \ref{Solimini}, is a semilinear operator $F(u) = -\Delta u - f(u)$ considered by Solimini (\cite{SOLIMINI}) for which a special right hand side $g$ has six preimages. Let $\phi_0$ be the ground state of the free (Dirichlet) Laplacian. As we shall see, from the min-max characterization of eigenvalues, lines of the form $u_0 + t \phi_0$ contain abundant bifurcation points and suffice to yield all solutions.

In the Appendix we handle inversion of segments in the neighborhood of the critical set $\cC$.  Instead of using variables  related to arc length \cite{ALLGOWER},
we work with spectral variables,  in the spirit of \cite{UECKER, CALNETO} and \cite{KAMINSKI}. In particular, we avoid the difficulty of partitioning Jacobians for discretizations of infinite dimensional problems.

We used $2 \times 2$ and MATLAB for graphs and numerical routines.

\bigskip

\noindent{\bf Acknowledgments}
We thank Nicolau Saldanha and José Cal Neto for extended conversations. Tomei gratefully acknowledges grants from FAPERJ (E-26/202.908/2018, E-26/200.980/2022) and CNPq (306309/2016-5, 304742/2021-0). Kaminski and Monteiro received graduate grants from CNPq and CAPES/FAPERJ respectively.

\section{Basic geometry} \label{model}

For real Banach spaces $X, Y$ and an open set $U \subset X$, we consider a  function $F: U \to Y$ with continuous derivative, which we usually restrict to closed domains $D \subset U$ with smooth boundary $\partial D$. Recall that the {\it critical set} $\cC$ of $F: U \to Y$ is
\[ \cC \ = \ \{ u \in U \ | \ F \ \hbox{is not a local homeomorphism from} \ u \ \hbox{to} \ F(u)\} \ .\]
Regular points are points of $U$ which are not critical. Since $F$ is of class $C^1$, if the Jacobian $DF(u): X \to Y, u \in U$, is an isomorphism, the inverse function theorem implies that $u$ is a regular point and $F$ is a local diffeomorphism of class $C^1$ at $u$. As $D \subset U$, any point in $D$ (in particular in $\partial D$) can be critical or not.

As in \cite{MST1} (and in \cite{DUCZMAL} for domains with boundary), define the {\it flower} $\mathcal{F} = F^{-1}(F(\Cc \cup \partial D))$,  a convenient description of the geometry of  $F$. As an example, Berger and Podolak \cite{BERGER} proved that, under the hypotheses of the Ambrosetti-Prodi theorem (Theorem \ref{APT}), the flower of  $F$ is essentially the simplest:  $\mathcal{F} = \cC$ and  is diffeomorphic to a topological hyperplane. 

In general, domain $D$ and codomain $Y$ split into {\it tiles}, the connected components of $D \setminus \mathcal{F}$ and $Y \setminus F(\cC \cup \partial D)$ respectively.  We assign a common label $A_i$ to a point and its image $A = F(A_i)$. 


As a first example of the geometric model in the Introduction, consider 
$$ F: U = D= X = \RR^2 \to Y =  \RR^2 \ , \quad (x,y) \mapsto  (x^2 - y^2 +x, 2 xy - y)\ . $$
\begin{figure} [ht] 
	\begin{centering}
		\includegraphics[scale=0.33]{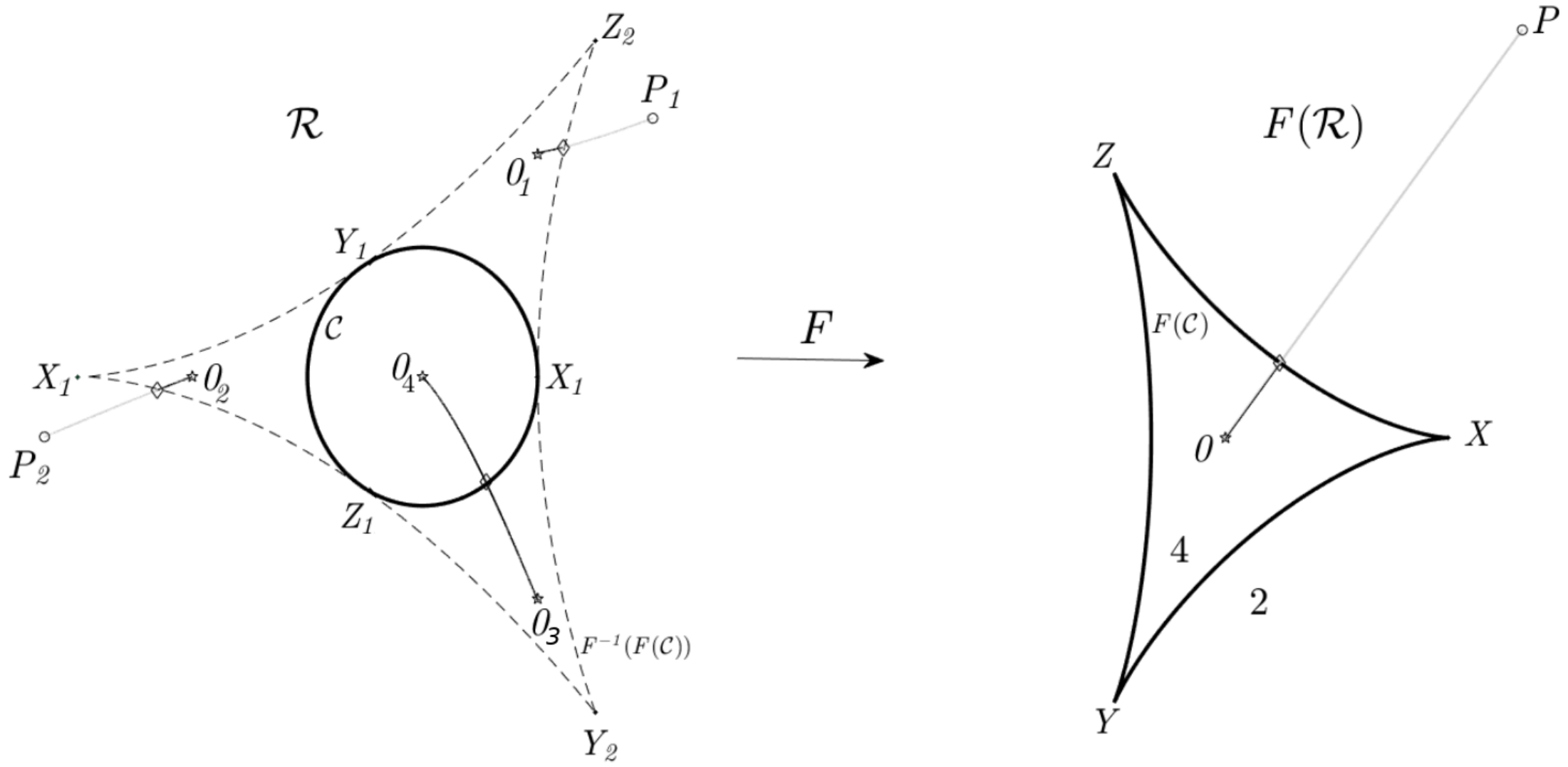}
		\caption{Five tiles in the domain, two in the codomain. }
		\label{fig:z2zbarra}
	\end{centering}
\end{figure}
Its domain\footnote{These pictures were obtained with $2\times 2$; horizontal and vertical scales are different.},
on the left of Figure \ref{fig:z2zbarra}, contains the critical set $\Cc$, a circle, and the flower $\mathcal{F}$, consisting of three curvilinear triangles\footnote{By a triangle, we mean the region bounded by three arcs.} with vertices $X_i, Y_i$ and $Z_i, i=1,2$. There are five tiles in the domain, two in the codomain. The map $F$ is a homeomorphism from each bounded tile  to the bounded tile $X_1Y_1Z_1$ in the image. Points in the image tile $XYZ$ have four preimages (in particular, $F(0) = 0$ has four preimages, indicated by $0_1, \ldots, 0_4$). The unbounded tile $\mathcal{R}$ in the domain is taken to  the unbounded tile $F(\mathcal{R})$, but the map is not bijective: each point in  $F(\mathcal{R})$ has {\it two} preimages, both in $\mathcal{R}$. More geometrically, all restrictions of $F$ to tiles are covering maps (\cite{HATCHER}), and thus must be diffeomorphisms when their image is simply connected. In Figure \ref{fig:z2zbarra}, the numbers on the tiles of the image are the (constant) number of preimages of points in each tile.

\subsection{Local theory} \label{basicfolds}

We consider the behavior of $F$ at points in $\cC$ and $\partial D$.

In Figure \ref{fig:z2zbarra}, points in different tiles in the image (necessarily adjacent, in this case, i.e., sharing an arc of images of critical points) have their number of preimages differing by two.
Inversion of the segment $[(0,0) , P]$ by continuation gives rise to two subsegments, whose interiors have two and four preimages. Starting from $P$ with initial condition $P_1$ or $P_2$, inversion carries through to $0$ without difficulties, giving rise to roots $0_1$ and $0_2$. However, when inverting from $0$ with initial conditions $0_3$ and $0_4$, continuation is interrupted. This is the expected behavior of inversion by continuation at a fold, as we now outline.

\begin{figure} [ht] 
	\begin{centering}
		\includegraphics[scale=0.3]{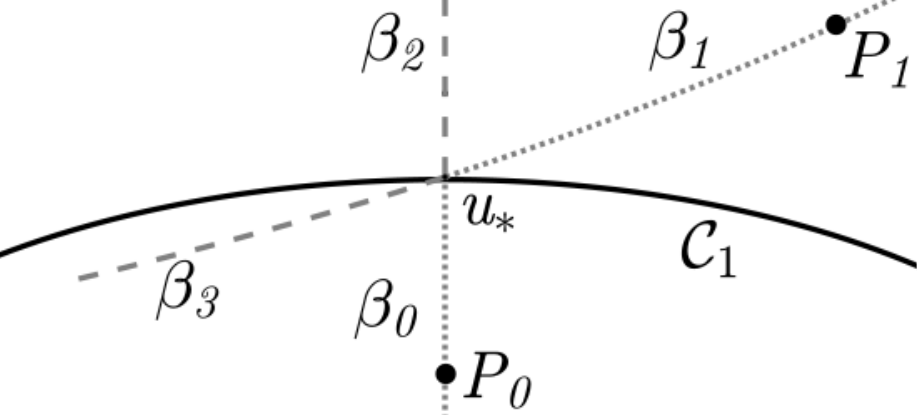} 
		\caption{Near a fold $u_\ast$.}
		\label{fig:extensao2}
	\end{centering}
\end{figure}

A critical point $x \in U$ is a {\it fold} of $F$ if and only if there are local changes of variables centered at $x$ and $F(x)$ for which $F$ becomes
\[ (t, z) \in \RR \times Z \mapsto (t^2, z) \in \RR \times Z\]
for some real Banach space $Z$.  On Figure \ref{fig:extensao2}, point $u_\ast$ is a fold, and the vertical line through it splits in two segments, $\beta_0$ and $\beta_2$. The inverse of $F(\beta_0)$ yields again $\beta_0$ and another arc $\beta_1$, also shown in Figure \ref{fig:extensao2}. In a similar fashion, the inverse of $F(\beta_2)$ contains $\beta_2$ and another arc $\beta_3$. Notice the similarity with a bifurcation diagram: two new branches $\beta_1$ and $\beta_3$ emanate from $r$.

\medskip
Folds are identified as follows. Let $X$ and $Y$ be real Banach spaces admitting a bounded inclusion $X \hookrightarrow Y$. Denote by $\cB(X,Y)$  the set of bounded operators between $X$ and $Y$, endowed with the usual operator norm. 

\medskip

\begin{theo} \label{trespassing}
	
	\noindent Suppose that the function $F: U \to Y$ is of class $C^3$. For a fixed $x \in U$, let the Jacobian $DF(x): X \to Y$ be  a Fredholm operator of index zero for which 0 is an eigenvalue. Let $\Ker DF(x)$ be spanned by a vector $k \in X$ such that $k \notin \Ran DF(x)$. Then the following facts hold.
	
	\noindent (1) For some open ball $B \subset  \cB(X,Y)$ centered in $DF(x)$, operators $T \in B$ are also Fredholm of index zero and $\dim \Ker T \le 1$. 
	
	\noindent (2) There is a  map $\lambda_s: T \in B \to \RR$ taking $T$ to its eigenvalue of smallest module, which is necessarily a real eigenvalue. The map is real analytic. For a suitable normalization, the  corresponding eigenvector map $T \mapsto \phi_s(T) \in X$ is  real analytic. 
	
	\noindent (3) For a small open ball $B_x \subset X$ centered in $x$, there are $C^3$ maps
	\[ \tilde x \in B_x \mapsto \lambda_s( DF(\tilde x)) \in \RR \ , \quad x \mapsto \phi_s( DF(\tilde x)) \in X\ . \]
	
	If additionally  
	$D \lambda_s (x) . \phi(x) \ \ne \ 0 $,  $x$ is a fold. 
	
\end{theo}

\medskip
In particular, the critical set $\cC$ of $F$ is a submanifold of $X$ of codimension one near $x$. Jacobians are not required to be self-adjoint operators. Characterizations of folds  in the infinite dimensional context may be found in \cite{BCT1, BCT2, D, BD, MST3}.

\begin{proof}  We barely sketch a lengthy argument. The implicit function theorem yield items (1) and (2), as described in Proposition 16 of \cite{CaTZ}. Item (3) then follows. Proposition 2.1 of \cite{MST3} implies item (4).	\end{proof}

\bigskip

We now consider points at the boundary $\partial D$.
As an example, restrict $F$ above to a disk $U$ with boundary $\partial D$ as in Figure \ref{fig:fronteira}. The flower now  includes $\partial D$ and an additional dotted curve $F^{-1} (F(\partial D))$, containing the preimages $A_i, B_i, C_i, i=1,2$ of points $A, B$ and $C$. Each one of the tiles I, II and III has a single preimage  inside $U$, the other being outside. Trespassing an arc of images of points of $\partial D$ only changes the number of preimages by one.

\begin{figure} [ht] 
	\begin{centering}
		\includegraphics[scale=0.35]{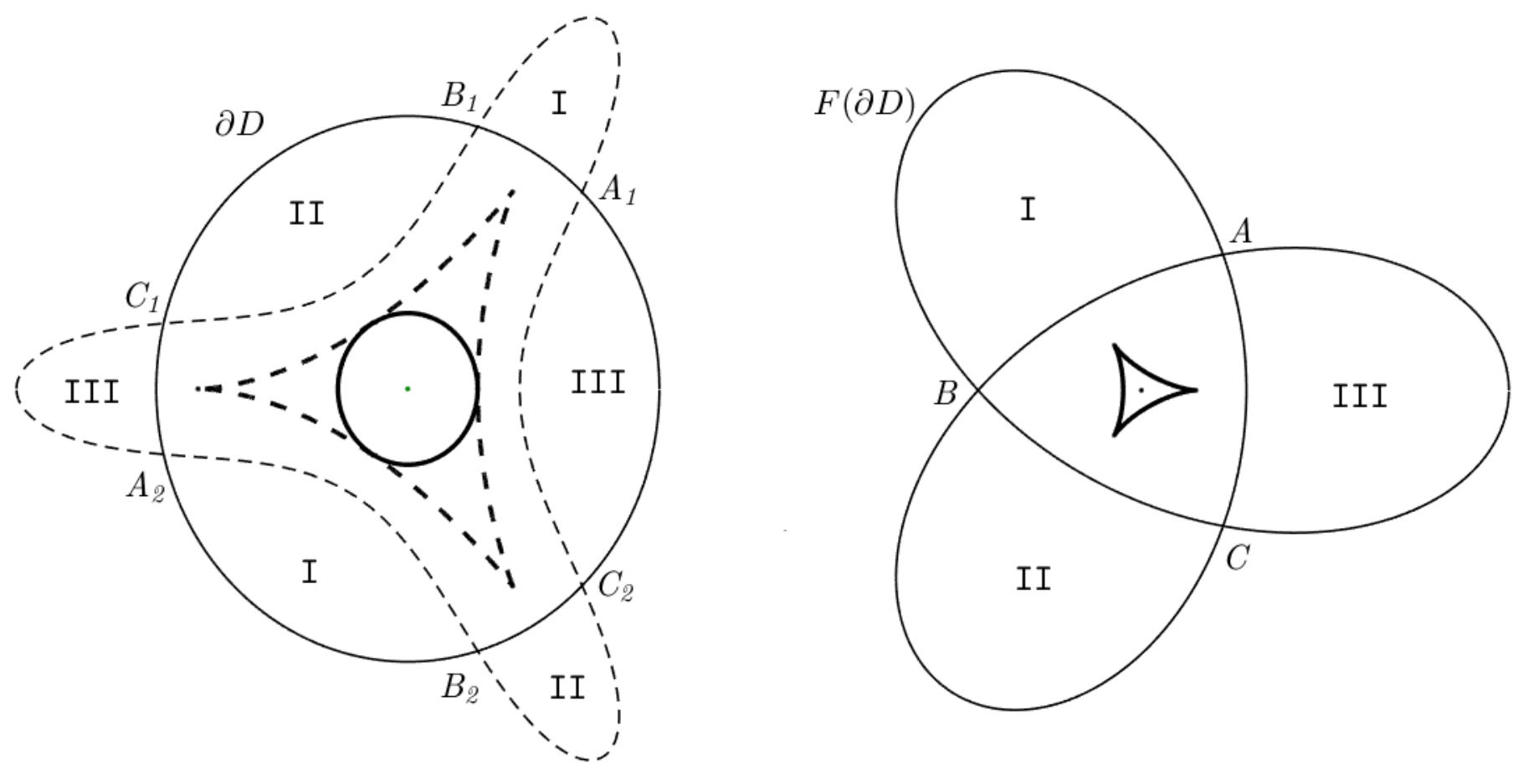}
		\caption{Behavior at the boundary. }
		\label{fig:fronteira}
	\end{centering}
\end{figure}

\subsection{Restricting functions to tiles, b-proper functions  } \label{primeiroexemplo}

A function $F: D \to Y$ is {\it proper} if the inverse of a compact set of $Y$ is a compact set of $D$. It is {\it proper on bounded sets} (or equivalently, b-proper) if its restriction $F_B: B \subset D \to Y$ to bounded, closed sets $B$ is proper.
As is well known, a continuous function $F: \RR^n \to \RR^n$ is proper if and only if $ \| F(x) \| = \infty$ as $\|x\| \to \infty$.  Proper functions are b-proper, but the second definition incorporates  a class of elliptic semilinear operators.

\medskip

\begin{prop} \label{gerais} Let $Y$ be a real Banach space and $G: Y \to Y$, $G(u) = u + \Phi(u)$, be continuous. Suppose that, for any closed ball $B \subset Y$, $\overline{\Phi(B)}$ is a compact set. Then $G$ is b-proper. Moreover,
	$G$ is proper if and only if $\| G(x) \| \to \infty$ as $\|x\| \to \infty$.
\end{prop}

\begin{proof} We prove that $G$ is b-proper, the remaining argument is left to the reader. For a sequence $y_n \in K$ with $y_n\to y_\infty \in K$, and $u_n \in B \subset Y$ such that $G(u_n) = u_n + \Phi(u_n) = y_n$. For an appropriate subsequence, $\Phi(u_{n_m}) \to w \in Y$ and then $u_{n_m}= - \Phi(u_{n_m}) + y_{n_m}\to -w + z_\infty \in Y$ and hence $B \cap G^{-1}(K)$ is compact.
\end{proof}

\begin{coro} \label{elipticos} Let $\Omega \subset \RR^n$ be a bounded set with smooth boundary. For a smooth $f: \RR\to \RR$, set $F:X = C^{2,\alpha}_D(\Omega) \to Y = C^{0,\alpha}(\Omega)$ given by $F(u) = - \Delta u + f(u)$. Then $F$ is b-proper  and
	$F$ is proper if and only if $\| F(x) \| = \infty$ as $\|x\| \to \infty$.
\end{coro}

\medskip
Here, $C^{2,\alpha}_D(\Omega) $ is the  H\"older space of functions equal to zero on $\partial \Omega$, $\alpha \in (0,1)$.



\begin{proof} Set $ G: Y \to Y, \ G(v) = v + \Phi(v)$, where $\Phi(v) = f(\Delta^{-1} v)$. Recall that $-\Delta: X \to Y$ is an isomorphism and the inclusion $X \hookrightarrow Y $ is compact. The nonlinear map $\Phi$  satisfies the hypotheses of the previous proposition (\cite{CHIAPPINELLI}).
\end{proof}

The geometric model of the Introduction holds for b-proper functions: here is the first step.

\begin{theo} \label{covers} For a bounded, open set $U \subset X$, consider a b-proper function $F: U \to Y$. Then the restriction of $F$ to a bounded tile $T_D \subset U \setminus \cF$ in the domain is a covering map  $F_D: T_D \to T_C = F(T_D)$  of finite degree, where $T_C$ is a tile in the image. Said differently, $F_D$  is a surjective local diffeomorphism and all points in $T_C$ have the same (finite) number of preimages.
\end{theo}

\medskip
Up to sign, the degree of $F_D$ is the number of the preimages under $F_D$ of any point in $T_C$. In particular, one may consider degree theory on restrictions of $F$ to tiles. In \cite{MST1} some relationships between the number of cusps and the degree of restrictions of $F$ to tiles are used in the search of critical curves.

Theorem \ref{covers} is a standard argument in the theory of covering spaces \cite{HATCHER}. For a proper function $F: X \to Y$, it holds also for unbounded tiles $T_D$ in the domain.

\begin{proof}
	
	Since $T_C$ is an open connected set, surjectivity follows once we prove that $F(T_D) \subset T_C$ is open and closed. Openness is clear from the inverse function theorem, as every point of $T_D$ is regular. We now show that $F(T_D)$ is closed. Take a convergent sequence $y_n = F(x_n) \in T_C$, $y_n \to y_\infty \in T_C$. As $F$ is a proper bounded function, there is a subsequence $x_{n_m} \in \overline{T_D}$ such that $x_{n_m} \to x_\infty \in \overline{T_D}$ and $F(x_\infty)=  y_\infty$. As $y_\infty \in T_C$, it is not in $F(\cC \cup \partial D$), and thus  $x_\infty \in T_D$.
	
	Consider the restriction $F_D: T_D \to T_C = F(T_D)$: we  show that $y \in T_C$ has finitely many preimages. By properness of $F_D$, $F^{-1}(y)$ is a compact set of $T_D$: if $F^{-1}(y)$ is infinite, it has an accumulation point $x_\ast$ for which $F(x_\ast)=y$. Also, $x_\ast \notin \mathcal{C}$, as $y \in T_C$. But at regular points, $F$ is a local homeomorphism: there are no convergent sequences to $x_\ast$ in $F^{-1}(y)$. 
	
	Since $F_D$ is a surjective local homeomorphism and each point has finitely many preimages,  it is a covering map: as $T_D$ is connected, the fact that all points in $T_C$ have the same number of preimages follows. We give details for the reader's convenience. By connectivity, it suffices to show that points sufficiently close to $y \in T_C$ have the same number of preimages.
	Suppose then $F^{-1}(y) = \{x_1, \ldots, x_k\} \subset T_D$, a collection of regular points: there must be sufficiently small, non- intersecting neighborhoods $V_{x_i}, i=1, \ldots,k$ of the points $x_i$ and $V_y \subset T$, such that the restrictions $F:V_{x_i} \to V_y$ are homeomorphisms. Thus, points in $V_y$ have at least $k$ preimages. Assume by contradiction  a sequence $y_n \to y, y_n \in V_y$, such that each  $y_n$ admits  an additional preimage $x_n^\ast$, necessarily outside of $\cup_i V_{x_i}$. By properness, for a convergent subsequence $x_{m_n}^\ast \to x_\infty^\ast$, $F(x_\infty^\ast)= y$, and $x_\infty^\ast \ne x_i, i=1,\ldots,k$.
\end{proof} 

A point $y \in F(\partial D)$ is a {\it generic critical value} if its preimages are regular points together with a single point  $x \in \cC$ which is a fold. Similarly,
$y \in F(\partial D)$ is a {\it generic boundary value} if its preimages contain regular points and a single point  $x \in \partial D$ for which $F$ extends as a local homeomorphism in an open neighborhood of $x \in X$. 

The following result completes the validation of the geometric model.

\medskip

\begin{theo} \label{adjacenttiles}
	Let $F$ be as in Theorem \ref{covers}. Suppose two bounded tiles in the image of $F$ have a common point $y$ in their boundaries which is a generic critical value (resp. a generic boundary value). Then the number of preimages of points in both tiles differ by two (resp. by one). 
\end{theo}

\medskip
The result admits natural extensions. The image of the function $F: \RR^2 \to \RR^2$ given by $F(x,y) = (x^2, y^2)$ is the positive quadrant, and leaving it through a boundary point  different from the origin implies a change of number of preimages equal to four: in a nutshell, the boundary point has two preimages which are folds.

\bigskip
Additional hypothesis naturally hold for usual functions, by arguments in the spirit of Sard's theorem. At a risk of sounding pedantic, 
examples are the following.

\smallskip \noindent (H1) $\cC, \partial D, F(\cC)$ and $F(\partial D)$ have empty interior.

\smallskip \noindent (H2) A dense subset of $\cC$ consists of folds.

\smallskip

Instead of verifying such facts in the examples, we take the standard approach in numerical analysis: one proceeds with the inversion process and accepts an occasional breakdown.

\section{Visualizable applications of the algorithm} \label{Visual}
The image of the function
$$ F: \mathbb{R}^2 \to  \mathbb{R}^2 \ , \quad (x,y) \mapsto (\cos (x) - x^2\cos (x) + 2x \sin (x), y)$$
is well represented by  a piece of cloth pleated  along vertical lines.
As indicated in Figure \ref{figbabado}, the critical set $\mathcal{C}=\{(k\pi,y), \  k \in \ZZ , \ y \in \mathbb{R}\}$ and its image $F(\cC)$ consist of vertical lines. All critical points are folds.  Points $p$ in the image are covered a different number of times, the number of  preimages of $p$. 
\begin{figure} [ht] 
	\begin{centering}
		\includegraphics[scale=0.30]{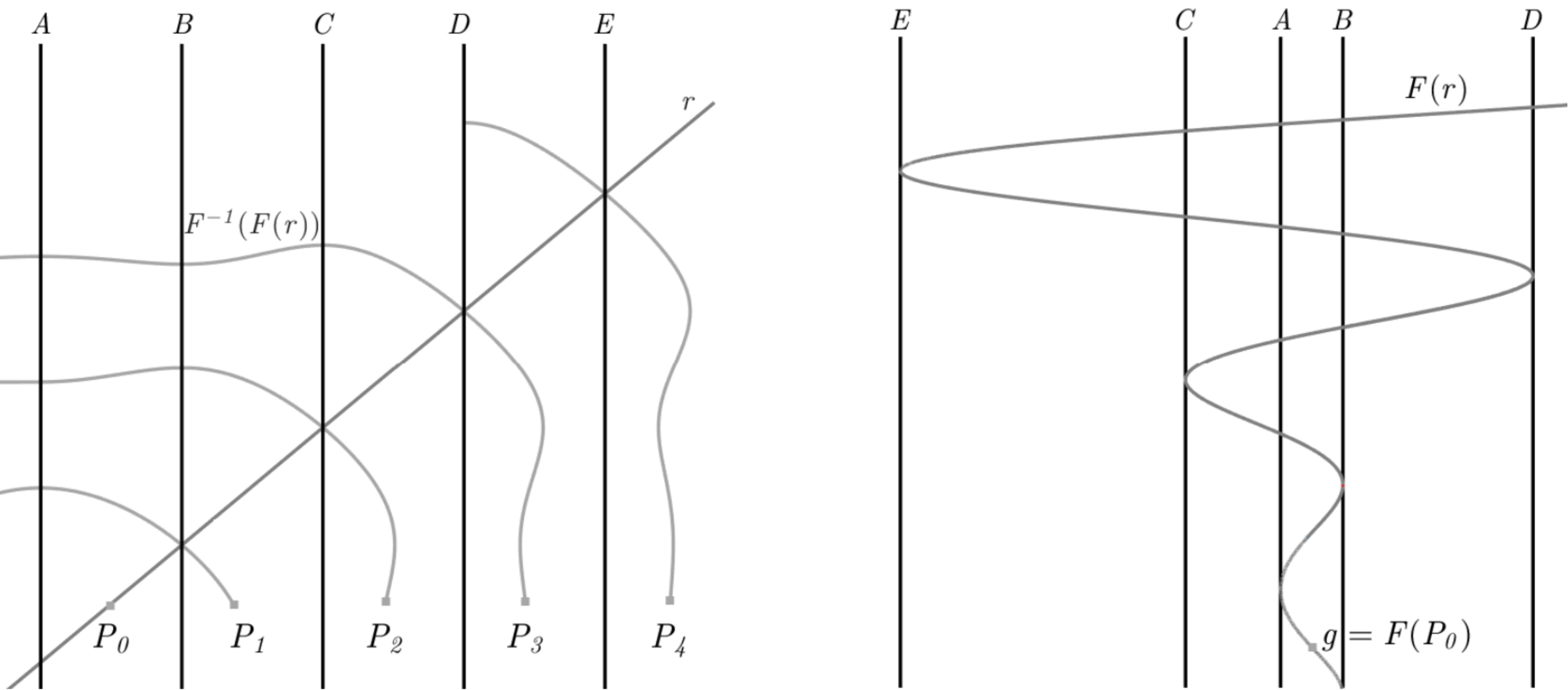}
		\caption{Pleats, the bifurcation diagram $\mathcal{B}$ and many preimages $P_i$ of $g=F(P_0)$.  }
		\label{figbabado}
	\end{centering}
\end{figure}

In Figure \ref{figbabado},  the line $r$ passes through $P_0$, a preimage of $g \in Y$. The image $F(r)$ oscillates among images of critical lines. Inversion of $F(r)$ yields $\mathcal{B}$, the {\it bifurcation diagram associated with} $r$ from $P_0$, the connected component of $F^{-1}(F(r))$ containing $P_0$.  Bifurcation points at the intersection of $r$ and $\mathcal{C}$ in turn gives rise to other preimages $P_i$ of $g=F(P_0)$. The line $r$ is chosen so as to intersect $\mathcal{C}$ abundantly. 

\bigskip
Consider now the smooth (not analytic) function
$$ F:\CC \to \CC \ , \quad z \mapsto  z^3 +\frac{5}{2}\ {\overline{z}}^2+z \ . $$
The critical set $\cC$ consists of the two curves $\mathcal{C} _1$  and $\mathcal{C} _2$ in Figure \ref{fig:estrela}, which roughly bound the three different regimes of the function: $z \sim 0, 1$ and $\infty$, where $F$ behaves  like $z$, $ {\overline{z}}^2$ and $z^3$ respectively. The  three  regimes already suggest that a line through the origin must hit the critical set at least four times.

We count preimages with Theorem \ref{adjacenttiles}.  From its behavior at infinity, $F$ is proper. Points in the unbounded tile in the image have three preimages, as for $z \sim \infty$, the function is cubic. The  unbounded tile in the domain covers the unbounded tile in the image of the right hand side  three times. Each of the five spiked tiles has five preimages, and the annulus surrounding the small triangle, seven. Finally, the interior of the small triangle has nine preimages: $F$ has  nine zeros. The flower,  in Figure \ref{fig:novinha}, illustrates these facts.

\begin{figure} [ht] 
	\begin{centering}
		\includegraphics[height=150pt,width=370pt]{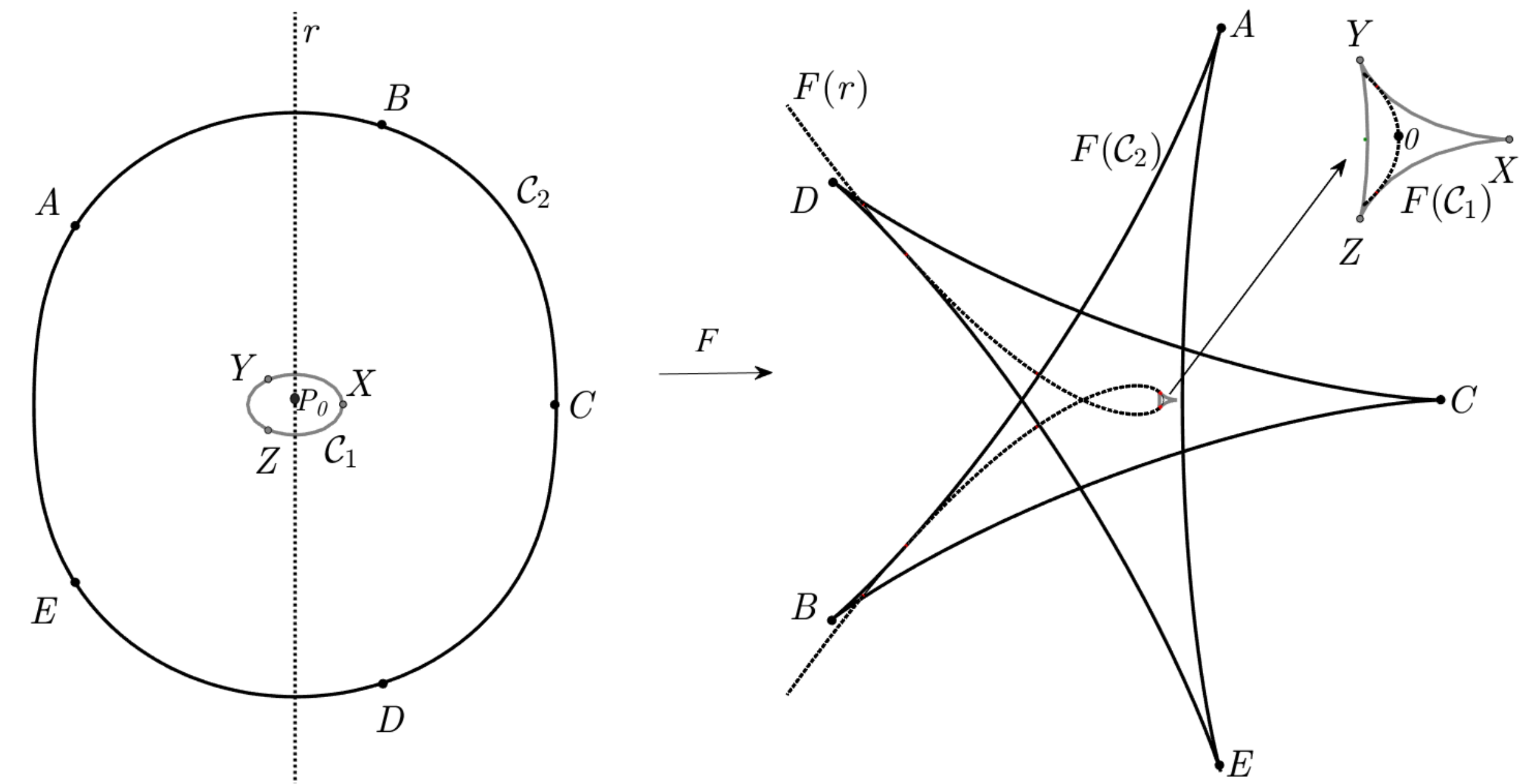}
		\caption{The critical set $\cC$, the line $r$ and their images.}
		\label{fig:estrela}
	\end{centering}
\end{figure}

\begin{figure} [ht] 
	\begin{centering}
		\includegraphics[height=150pt,width=370pt]{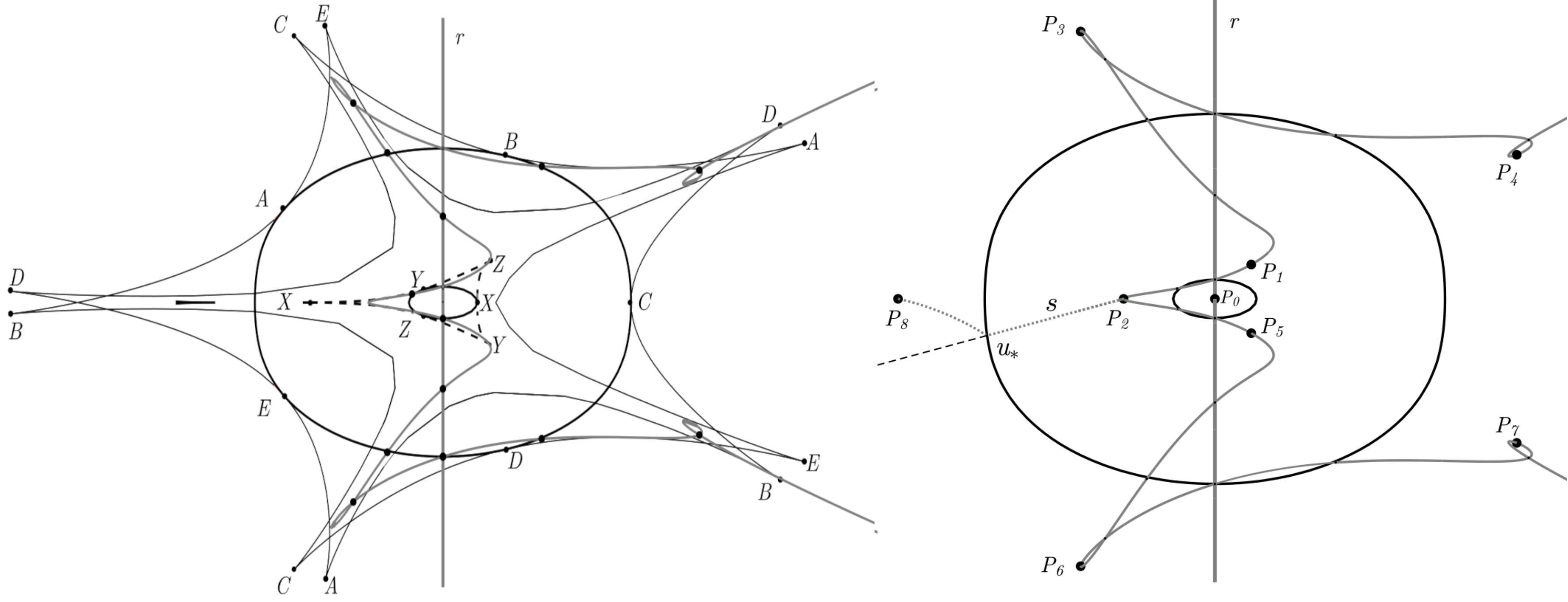}
		\caption{On the left, $\cC, r,  \mathcal{F}$ and $\mathcal{B}$. On the right, $r, \cC, \mathcal{B}$ and an extension yielding a zero $P_8$.}
		\label{fig:novinha}
	\end{centering}
\end{figure}



Let $r$ be the vertical axis, $P_0=(0,0) \in r$. Figure \ref{fig:estrela} shows $r$ and $F(r)$ and
Figure \ref{fig:novinha}  the  flower $ \mathcal{F} = F^{-1}(F(\mathcal{C}_1)) \cup F^{-1}(F(\mathcal{C}_2))$: dotted black lines and $\cC_1$ form  $F^{-1}(F(\mathcal{C}_1))$, while $F^{-1}(F(\mathcal{C}_2))$ consists of continuous black lines.  Amplification of $ \mathcal{F}$ shows  five thin triangles, one in each `petal' (one is visible in the petal on the left), each a full preimage of the triangle $XYZ$ in the image (the enlarged detail in Figure 	\ref{fig:estrela}). Add the other four preimages at tiles bounded by $F^{-1}(F(\mathcal{C}_1))$ to spot the nine zeros of $F$.  The flower is computationally expensive and is not computed in the algorithm we present.

On the left of Figure \ref{fig:novinha} are shown $\cC$, $\mathcal{F}$, $r$ and $\mathcal{B}$, the bifurcation diagram associated with $r$ from $P_0=(0,0)$. To emphasize $r$ and $\mathcal{B}$, we removed $\mathcal{F}$ on the right of Figure \ref{fig:novinha}. The sets $r$ and  $\mathcal{C}$ meet at four folds. The set $\mathcal{B}$ contains the line $r$ and eight of the nine zeros of $F$. The missing zero, $P_8$, is on the petal on the left of Figure \ref{fig:novinha}.


The  branches originating from the four critical points in $r$ lead by continuation to additional zeros $P_i$ of $F$.  What about the missing zero? For the half-line $s$ joining $P_2$ to infinity in  Figure \ref{fig:novinha}, bifurcation at $u_*\in F^{-1}(F(s))$  obtains $P_8$. For completeness,
$$
\begin{array}{c}
	P_1=( 0.2141, 0.3313) \ , \quad P_2 = (-0.5367,0.0000)\ , \quad P_3=( -0.7893, 2.5802)\ , \\
	\quad P_4=(1.7752,1.3903)\ ,  \quad P_5=( 0.2141,-0.3313)\ , \quad P_6=(-0.7893, -2.5802)\ ,\\
	P_7=(1.7752,-1.3903)\ , \quad P_8=(-1.8633,0.0000) \ .
\end{array}
$$

\subsection{Singularities and global properties of $F$} \label{winding}

\medskip
The example above shows that  curves $r$ with different intersection with the same critical component of $\cC$ may yield different zeros.  This will happen in other examples in the text: curves will be parallel lines, each generating a set of zeros. We provide an explanation for this fact based on the geometric model.

This is already visible for the simpler function $F$ in Figure \ref{fig:z2zbarra}.  Informally, $F$ at a fold point $p$ of $\cC$ looks like a mirror, in the sense that points on both sides of the arc $\cC_i$ of $\cC$ near $p$ are taken to the same side of the arc $F(\cC_i)$ near $F(p)$, as shown in Figure \ref{fig:extensao2}. In the example in Figure \ref{fig:z2zbarra}, at the three points  which are not folds, one might think of adjacent broken mirrors.  Close to the image $X = F(X_1)$ of such points, say $X_1$, there are points with three preimages close to $X_1$. 

In higher dimensions (in particular, infinite dimensional spaces), there are critical points at which arbitrarily many splintered mirrors coalesce. At such higher order singularities, there are points with clusters of preimages.  H. McKean  conjectured the existence  of arbitrarily deep singularities for the operator $F(u) = - u'' + u^2$, where $u$ satisfies Dirichlet conditions in $[0,1]$, and the result was proved in \cite{ARDILA}.

The original algorithm, $2 \times 2$, identifies and explore cusps, a higher order singularity, in the two dimensional context. The  lack of information about higher order singularities forces the current algorithm to perform more searches, by essentially choosing curves $r$ intercepting the critical set along different mirrors.

The region surrounded by $\cC$ in Figure \ref{fig:z2zbarra} has three mirrored images.
The presence of higher singularities is related to the fact that  there may be tiles on which the covering map induced by the restriction $F$  is not injective, as in the case of the annulus tile in the domain of $F$ of this section. This in turn requires the image tile to have nontrivial topology (more precisely, a nontrivial fundamental group), as simply connected domains are covered only by homeomorphisms.

\section{Discretized nonlinear Sturm-Liouville maps} \label{SturmLiouville}

We consider the nonlinear Sturm-Liouville operator between Sobolev spaces,
\begin{equation}
F:X=H^2([0,\pi]) \cap H^1_0([0,\pi]) \longrightarrow  Y=L^2([0,\pi]) \ , \quad F(u) = -u''-f(u)  \ .
\end{equation}
We compare two theorems which count solutions of $F(u) = g$ for special functions $g$ for the operator and its discretizations.
Recall that the linear operator $u \mapsto - u''$ acting on functions satisfying Dirichlet boundary conditions has eigenvalues equal to $\lambda_k = k^2, k=1, 2, \ldots$ and associated eigenvectors $\phi_k = \sin(kx)$.

\medskip
\begin{theo}[Costa-Figueiredo-Srikanth \cite{COSTA}]
Let $f:\mathbb{R} \to \mathbb{R}$ be a smooth,  convex function which is asymptotically linear with parameters $\ell_-$,  $\ell_+$,
\[
\lim_{x \to - \infty} f'(x) = \ell_-,  \quad \lim_{x \to \infty} f'(x) = \ell_+ \ ,\]
\[
\ell_- < \lambda_1 = 1, \ \ \lambda_k = k^2<\ell_+<(k+1)^2 = \lambda_{k+1}, \quad \ell_-, \ell_+ \notin \{k^2, k \in \mathbb{N}\} \ . \]
Then, for large $t > 0$, $F(u) = -t\sin(x) $
has exactly $2k$ solutions.
\end{theo}

\medskip

The standard discretization of $F$ is defined over the regular mesh
\[ \overline{I_h}=\left\{x_i = ih \ , \ i=0,...,n+1 \ ,   h=\frac{\pi}{n+1} \right\} \ .\]
As usual, for points $x_i \in I_h=\overline{I_h} \setminus \{x_0,x_{n+1}\}$, approximate
$$
-u''(x_i) \sim \frac{1}{h^2} \left(-u(x_{i+1})+2u(x_i)-u(x_{i-1})\right)
$$
and consider the tridiagonal,  $n \times n$ symmetric  matrix $A^h$ with diagonal entries $2/h^2$ and remaining nonzero entries equal to $-1/h^2$. Its (simple) eigenvalues are   \[ \lambda_k^h \ = \frac{2}{h^2} -  \frac{2}{h^2}\cos(\pi-kh)  , \ k = 1, \ldots, n \]
with associated eigenvectors $\phi_k^h = \sin(k I_h)= (\sin(kx_i))$ for $x_i \in I_h$.
Similarly, set $u^h = u(I_h)$ and $f(u^h)=(f(u(I_h))^T$. Finally, the discretized operator is 
\[ F^h: \mathbb{R}^n \to \mathbb{R}^n , u^h \mapsto  A^h u^h - f(u^h) . \]

\begin{theo}[Teles-Tomei \cite{TELES}] Let $f$ be smooth, convex, asymptotically linear function with parameters $\ell_-$ and $\ell_+$ satisfying
$  \ell_-< \lambda_1^h   \ ,  \lambda_k^h < \ell_+  < \lambda_{k+1}^h $.
Let $y,p \in \mathbb{R}^n $, $p_i > 0$. Then, for large parameters $t >0$ and $\ell_+ $, the equation
\[F^h(u^h)= A^hu^h-f(u^h)=y -tp,\]
has exactly $2^n$ solutions.
\end{theo}

As $h \to 0$, most solutions of $F(u) = g$ disappear. The discretized problem is an interesting test case for our algorithm:  $F^h(u)=g$ has abundant solutions.

The critical set of $F$ in both contexts has been studied in \cite{BURGHELEA} and \cite{TOMEIEBUENO}. {\it The  property below  is all it takes to implement the algorithm.} Notice that Jacobians in both cases consist of linear operators with eigenvalues  labeled in increasing order.

\medskip

\begin{theo} Lines in the domain with direction given by a positive function (vector, in the discrete case) intercept the critical set at points abundantly. Essentially, there is one intersection associated with one eigenvalue.
	\end{theo}

\medskip
This follows from min-max arguments on the Jacobians $DF(u)$ along such lines. The argument is given in a more general context --- for Jacobians given by semilinear elliptic operators ---  in  Proposition  \ref{Minmax}.

\medskip 
 Notice the affinity beween such lines and rays from the origin in the examples of Section \ref{Visual}. As in Section \ref{winding}, different curves might lead to a different set of solutions obtained by bifurcation and continuation.

\subsection{Piecewise linear geometry and a data base of solutions} \label{table}
\label{sec:TELES-TOMEI}

We follow \cite{COSTA} and \cite{TELES}, and consider a piecewise linear map
\begin{equation} \label{piecewise}
f(x) = \begin{cases}
	\ell_- x\ , \ x < 0   \\
	\ell_+ x\ , \ x > 0
\end{cases} \ .
\end{equation}
The associated function $F^h$ is globally continuous and linear when restricted to each orthant of $\mathbb{R}^n$. The lack of differentiability leads us to be careful with the identification of the critical set $\cC$, an issue we address in Section \ref{numericsSL}.

For the case of the piecewise linear function $f$, an explicit solution for the continuous problem was given by Lazer and McKenna \cite{LAZERANDMCKENNA}.
For $\ell_-<\lambda_1 = 1 <\ell_+$ and $t>0$, two solutions of $-u''-\ell_+ u^{+}+\ell_- u^{-}=-t\sin(x)$ are
\[\frac{ t\sin(x)}{\ell_+ -1} > 0  \ , \quad \frac{t\sin(x)}{\ell_- -1} < 0 \ . \]
For the discretized operator, it is easy to verify that two solutions are given by
\begin{equation} \label{LazerMcKenna}
	\frac{t\sin(I_h)}{\ell_+ -\lambda_1^h} \ \ \
\text{and} \ \ \
\frac{t\sin(I_h)}{\ell_- - \lambda_1^h},  \quad \lambda_1^h=\frac{2}{h^2} \left(1-\cos(h) \right) \ .
\end{equation}
Again, the entries of the first (resp. second) solution are positive (resp. negative).

\bigskip

Following \cite{TELES}, we consider lines $r$ aligned to the (normalized, positive) eigenvector $\phi_1^h$ associated to the smallest eigenvalue $\lambda_1^h$ of $A^h$. Set $V = \langle \phi_1^h \rangle$, $H = V^\perp$ and split  $\mathbb{R}^n = H \oplus  V$. The critical set $\mathcal{C}$ of the map $F^h: \mathbb{R}^n \to \mathbb{R}^n$ consists of $n$ hypersurfaces $\mathcal{C}_j, j=1, \ldots,n$.  
Each $\mathcal{C}_j$ is a graph of a function $c_j : H \to V$, and in particular projects diffeomorphically from $\mathbb{R}^n$ to $H$. Thus, topologically, $\mathcal{C}_j$ is trivial.
Each line $r$ intercepts all the critical components $\mathcal{C}_j$ of $F^h$.

The images $F(\mathcal{C}_j)$ are more complicated: for $\ell_+ >> 0$, $F(\mathcal{C}_j)$ wraps around the line
$\{t(1, 1, \ldots, 1), t \in \mathbb{R} \}\subset \mathbb{R}^n$ substantially: $\binom{n-1}{j-1}$ times! It is this  geometric turbulence which gives rise to the abundance of preimages for appropriate right hand sides. Said differently, some tiles in the image have nontrivial fundamental group and are abundantly covered. The reader should compare this situation with the second example in Section \ref{Visual}: cubic behavior at infinity leads winding of the image of the outermost critical curve and to nine zeros of $F$.

\bigskip

Set $n = 15$, $g^h = -1000 \sin(I_h)$, $f$ as in (\ref{piecewise}) with asymptotic parameters

\[ \ell_-=\frac{\lambda_{1}^h}{2} \ \ \text{and} \ \  \ell_{+}^k=\frac{\lambda_k^h+\lambda_{k+1}^h}{2}, \ \text{for} \ k=1,...,14 \ ,  \ \ \ell_{+}^{15}=\lambda_{15}^h+\frac{\lambda_{1}^h}{2}.\]
The number of solutions $N$ for different $\ell_+^k$ is given below \cite{TELES}.

\begin{table}  [ht] 
\small
\centering
\begin{tabular}{|c|c|c|c|c|c|c|c|c|c|c|c|c|c|c|c|}
	\hline
	$k$ &1&2&3&4&5&6&7&8&9&10&11&12&13&14&15  \\ \hline
	$N$ & 2 & 4 & 6 & 8 & 12 & 12 &22 & 24 & 32 & 100 & 286 &634 & 972 & 1320 & 2058\\ \hline
\end{tabular}
\label{Ng}
\end{table}

For small $k$, $N(g^h)$ is of the form $2k$, as for the continuous counterpart. At some point, $N(g^h)$ becomes exponential, $2^k$. The numerics in \cite{TELES} is elementary: solve a linear system in each orthant and check if the solution  belongs to the orthant, generating a reliable data bank of solutions.  Here, we compare solutions obtained by inducing bifurcations with those in the data bank. Another computational alternative, the shooting method for the discretized problem, did not perform well: we intend to explore the issue in a forthcoming paper.

\subsection{The critical set of $F^h$ for a piecewise linear $f$} \label{numericsSL}

We now relate criticality and the piecewise linear nature of the function $f$.
An (open) orthant $\cO \subset \RR^n$ is defined by the sign of its vector entries. For $v \in \cO$, $f(v)$ is multiplication by a diagonal matrix $D^\cO$, with diagonal entries equal to $\ell_-$ or $\ell_+$, according to the sign of the associated entry of $v$. Thus, the continuous map $F^h: \RR^n  \to \RR^n$ is of the form $F^h = A^h - D^{\cO}$ when restricted to $\cO$.

Generically (i.e., for an open, dense set of pairs $(\ell_-, \ell_+)$), all the matrices $A^h - D^{\cO}$ are invertible. Assuming this, in each orthant $\cO$ the (linear) map is trivially a local diffeomorphism. The critical set of $F^h$ makes (topological) sense: it is the set in which $F^h$ is not a local homeomorphism.
Say a vector $v$ is in the boundary of exactly two orthants, $\cO_1$ and $\cO_2$ (equivalently, $v$ has a single entry equal to zero). If $\sign \det (A^h - D^{\cO_1}) = \sign \det (A^h - D^{\cO_1})$, the map $F^h$  is a local homeomorphism at $v$. If instead $\sign \det (A^h - D^{\cO_1}) =  -\sign \det (A^h - D^{\cO_1})$, $F^h$ near $v$ is a topological fold.  The critical set consists of pieces of coordinate planes between two orthants in which $\det (A^h - D^{\cO})$ changes sign. In the nongeneric situation, whenever $A^h - D^{\cO}$ is not invertible, we  include $\cO$ in the critical set.

The critical set is connected. A (generic) line $r$ parallel to the vector $\phi_1^h$  intercepts $\mathcal{C}$ in as many eigenvalues of $A^h$ there are between $\ell_-$ and $\ell_+$.

\subsubsection{The visualizable case,  $n=2$} \label{visualLS}

Due to the lack of differentiability of $f$, we cannot use Newton's method as a local solver. We present the alternative approach in a visualizable example. If $n=2$, we have $h = \pi/3$ and the  matrix $A^h$
has eigenvalues  $\lambda_1 \approx 0.9119$ and $\lambda_2 \approx 2.7357$. For the values $\ell_-=-1$ and $\ell_+=2$ or $\ell_+=4$, $F^h = A^h - D^{\cO}$ in the interior of each quadrant is invertible.
Figure 7 
describes the critical sets and their images for the two values of $\ell_+$. In each quadrant of the domain we indicate $\sign \det (A^h - D^\cO)$. In each connected component of  $\mathbb{R}^2 \setminus F^h(\mathcal{C}) $ we specify instead the number of preimages. The fact that $\cC$ for $\ell_+ = 4$ is not a (topological) manifold is innocuous. 

\begin{figure}  [ht] 
\centering
\begin{minipage}{0.45\textwidth}
	\centering
	\includegraphics[width=1.0\textwidth]{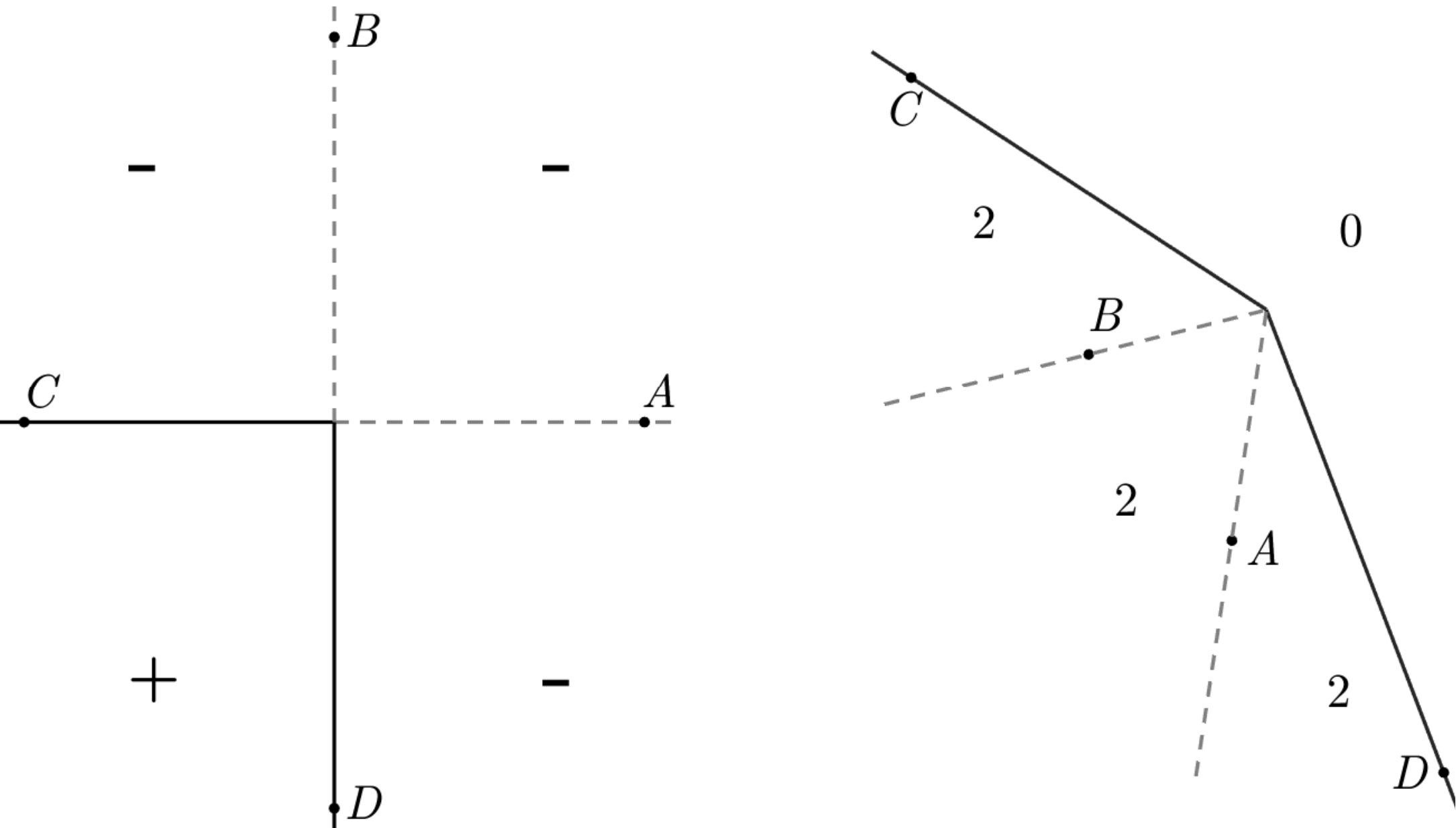} 
	\label{fig4a}
\end{minipage}
\begin{minipage}{0.45\textwidth}
	\centering
	\includegraphics[width=1.0\textwidth]{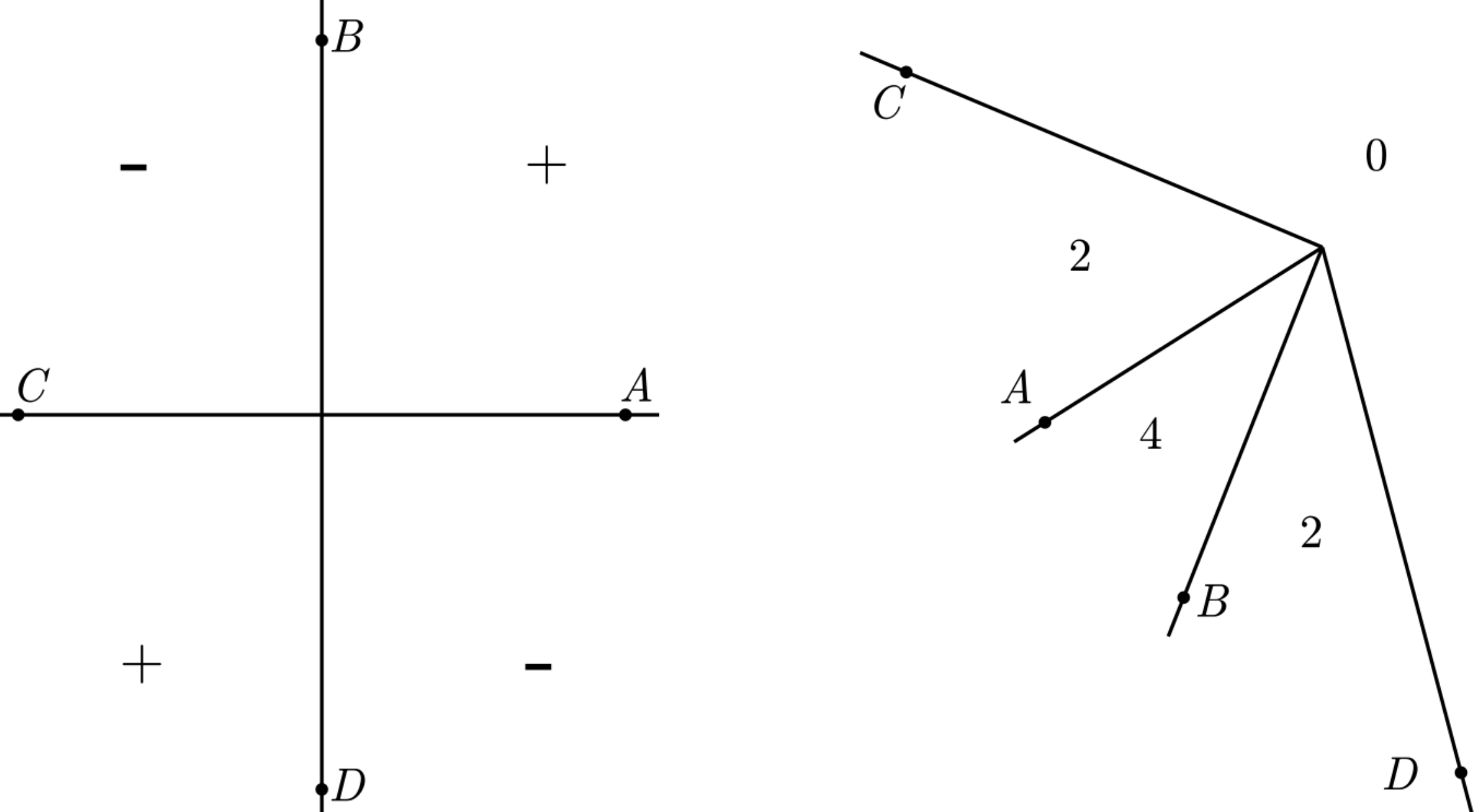} 
\end{minipage}
\hfill

\caption{$\mathcal{C}$ and $F^h(\mathcal{C})$ for  $\ell_-=-1, \ell_+=2$ and for $\ell_-=-1, \ell_+=4$}
\end{figure}

\begin{figure}  [ht] \label{fig4b}
\centering
\begin{minipage}{0.45\textwidth}
	\centering
	\includegraphics[width=0.9\textwidth]{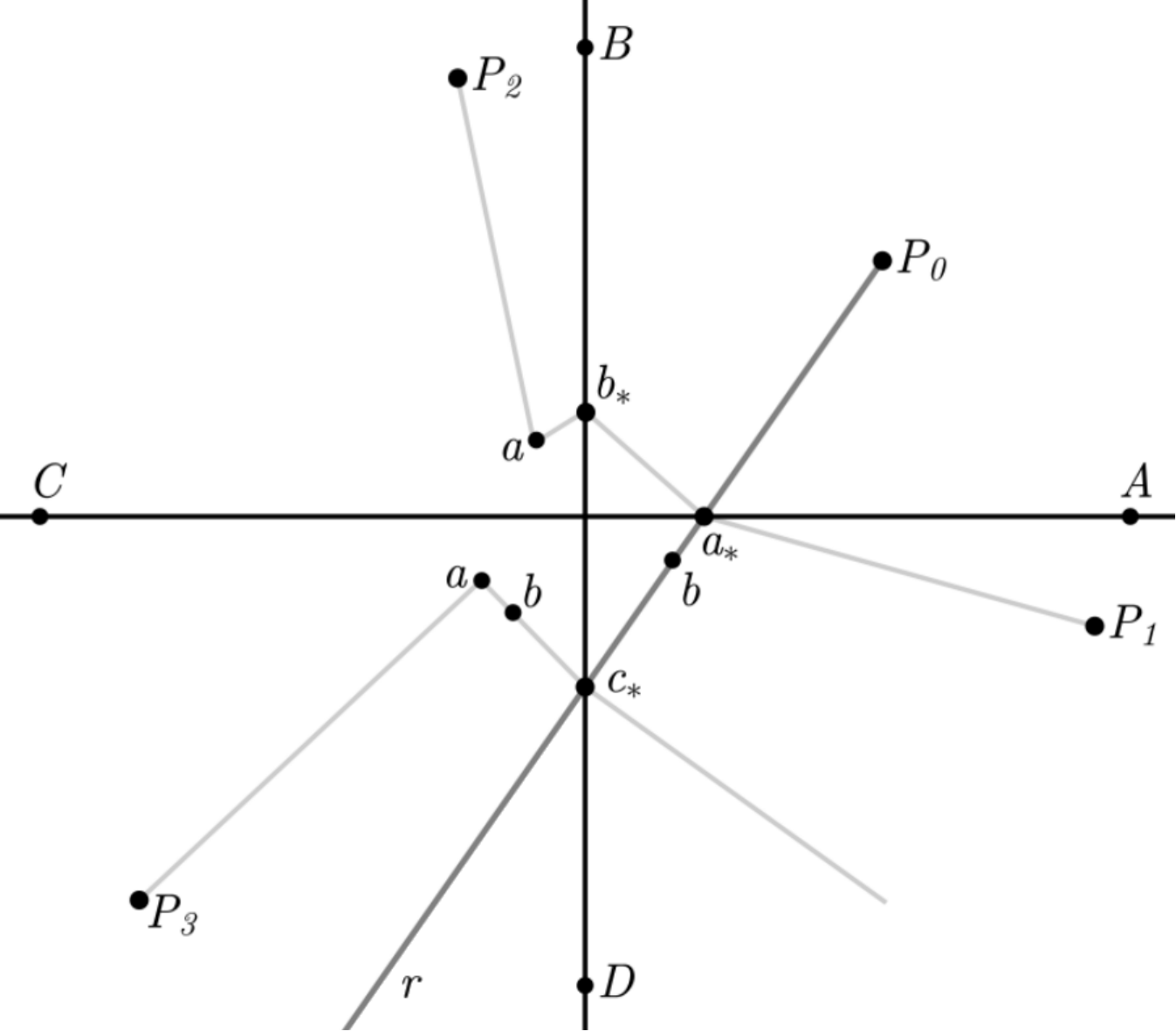} 
\end{minipage}
\begin{minipage}{0.45\textwidth}
	\centering
	\includegraphics[width=0.75\textwidth]{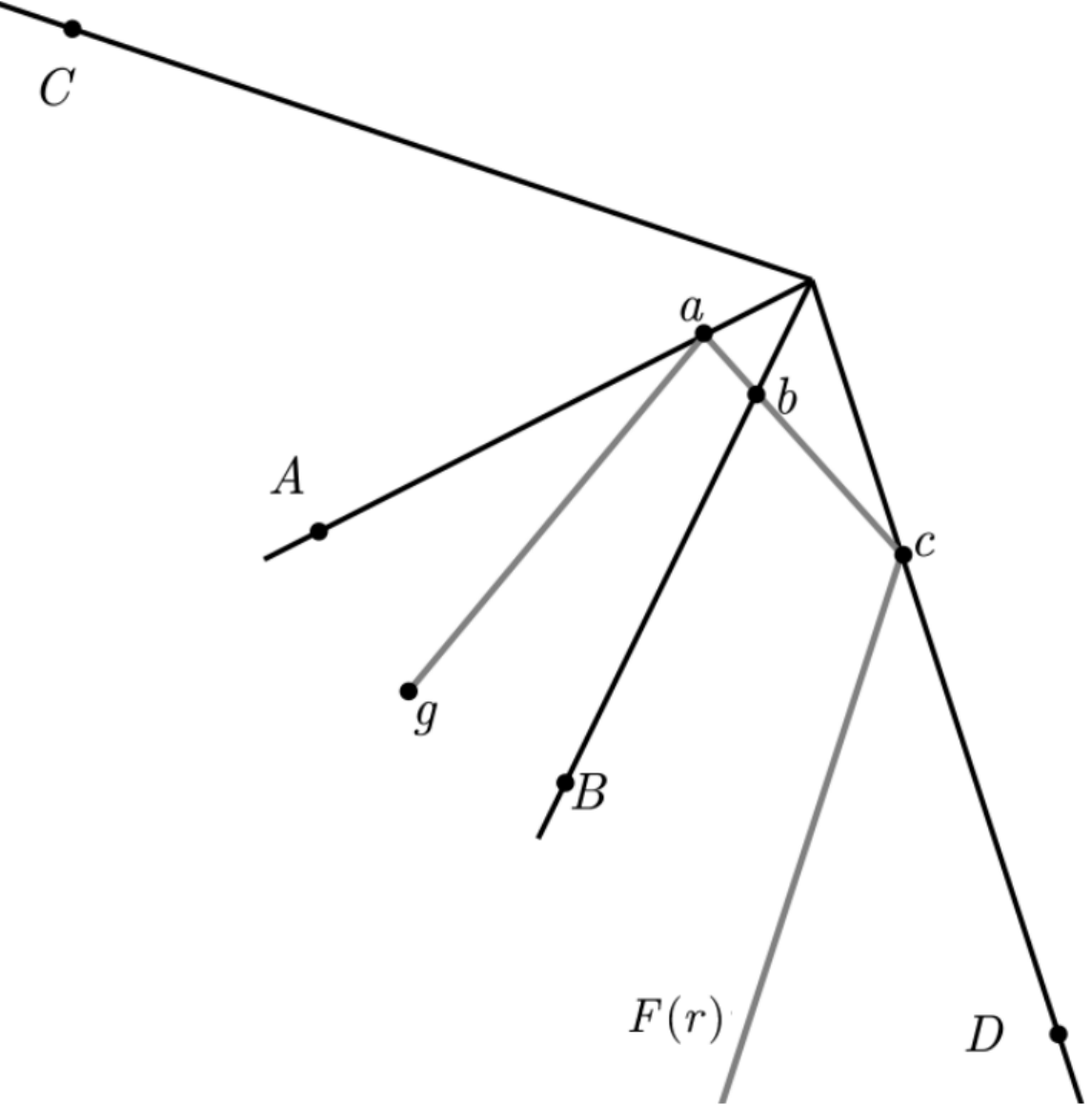} 
\end{minipage}
\caption{$\mathcal{B}$ (left) and $F^h(r)$ (right)  for $\ell_-=-1, \ell_+=4$}
\hfill

\end{figure}

%
%
We consider $\ell_-=-1$,  $\ell_+=4$.
To invert $g=-1000\sin(I_h)$, we start with the positive solution
$P_0 \approx ( 280.4396, 280.4396)$ (as in Section \ref{numericsSL}) and draw the half-line
$r =\{P_0+s(0.2\sin(2I_h)-0.8\sin(I_h)), \  s \geq 0\} \subset \mathbb{R}^{2}$ through $P_0$.
We add the term  $0.2\sin(2I_h)$ to minimize the possibility that $r$ intercepts $\cC$ at points which are not (topological) folds, i.e., points with more than one entry equal to zero (this is especially relevant for $n$ large). Figure 8 
shows $r$, $F^h(r)$ and the bifurcation diagram $\mathcal{B}$, the connected component of $F^{-1}(F(r))$ containing $P_0=(0,0)$.
The critical points of  $F^h$ in $r$ are $a$ and $c$. At these points, $F^h$ is a fold, and parts of $r$ get mirrored along the critical set. Its preimages are obtained by solving linear systems. They in turn may intercept $\cC$ (at $b$) and, by continuation, yield the missing three solutions $P_i$.

\subsection{Discretizing with $n=15$ points}

We consider two situations. In the first, the asymptotic values $\ell_-=0.4984$ e $\ell_+=19.1248$ enclose four eigenvalues of $A^h$ ($k=4$). 
A positive solution of $F^h(u) = g^h=-1000\sin(I_h)$ is $P_0=55.1633\sin(I_h)$, from Section \ref{numericsSL}. We use half-lines which are perturbations of the eigenvector $\sin(I_h)$ associated with the smallest eigenvalue,
\[ r=\{P_0 + s(\sin(I_h)-0.1\sin(2I_h) -0.1\sin(3I_h)- 0.1\sin(4I_h)), \ \text{com} \ s \geq 0\} \subset \mathbb{R}^{15} \ .\] 
to ensure simple, transversal intersections of $r$ with the critical set $\mathcal{C}$ of $F^h$. Figure  \ref{fig:4.7} represents  five open components $R_k, k=1, \ldots,5$ of the regular set, $\RR^{15} \setminus \mathcal{C}$. In each $R_k$, the sign of  $\det(A^h - D^O)$ is constant.
The bifurcation diagram $\mathcal{B}$ associated with $r$ from $P_0$  contains all the preimages in the data bank.

\begin{figure}  [ht]
\centering
\begin{minipage}{0.45\textwidth}
	\centering
	\includegraphics[width=165pt]{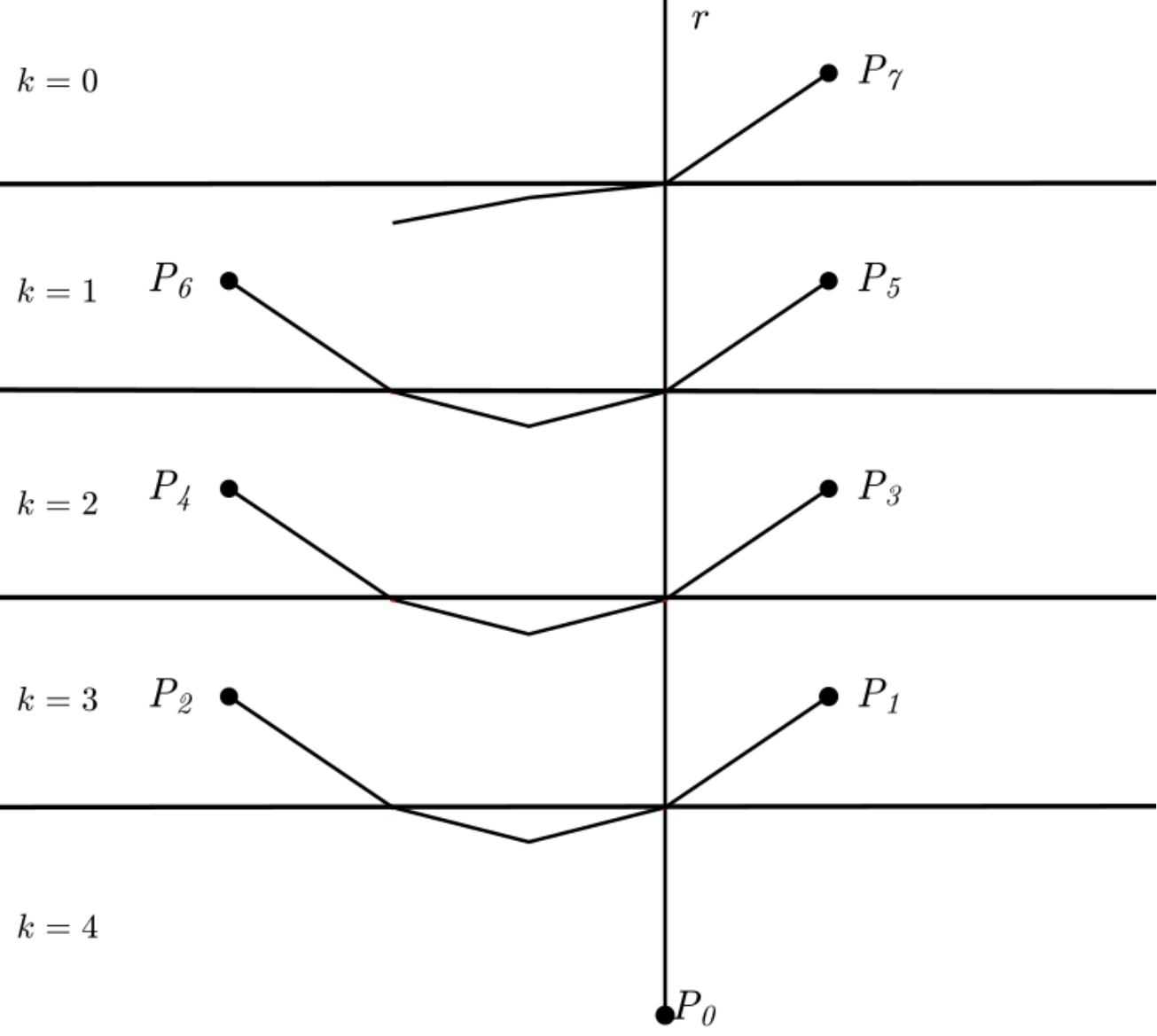}
	\caption{$\mathcal{B}$ for $k=4$.}
	\label{fig:4.7}
\end{minipage}
\hspace{0.5cm}
\begin{minipage}{0.45\textwidth}
	\centering
	\includegraphics[width=165pt]{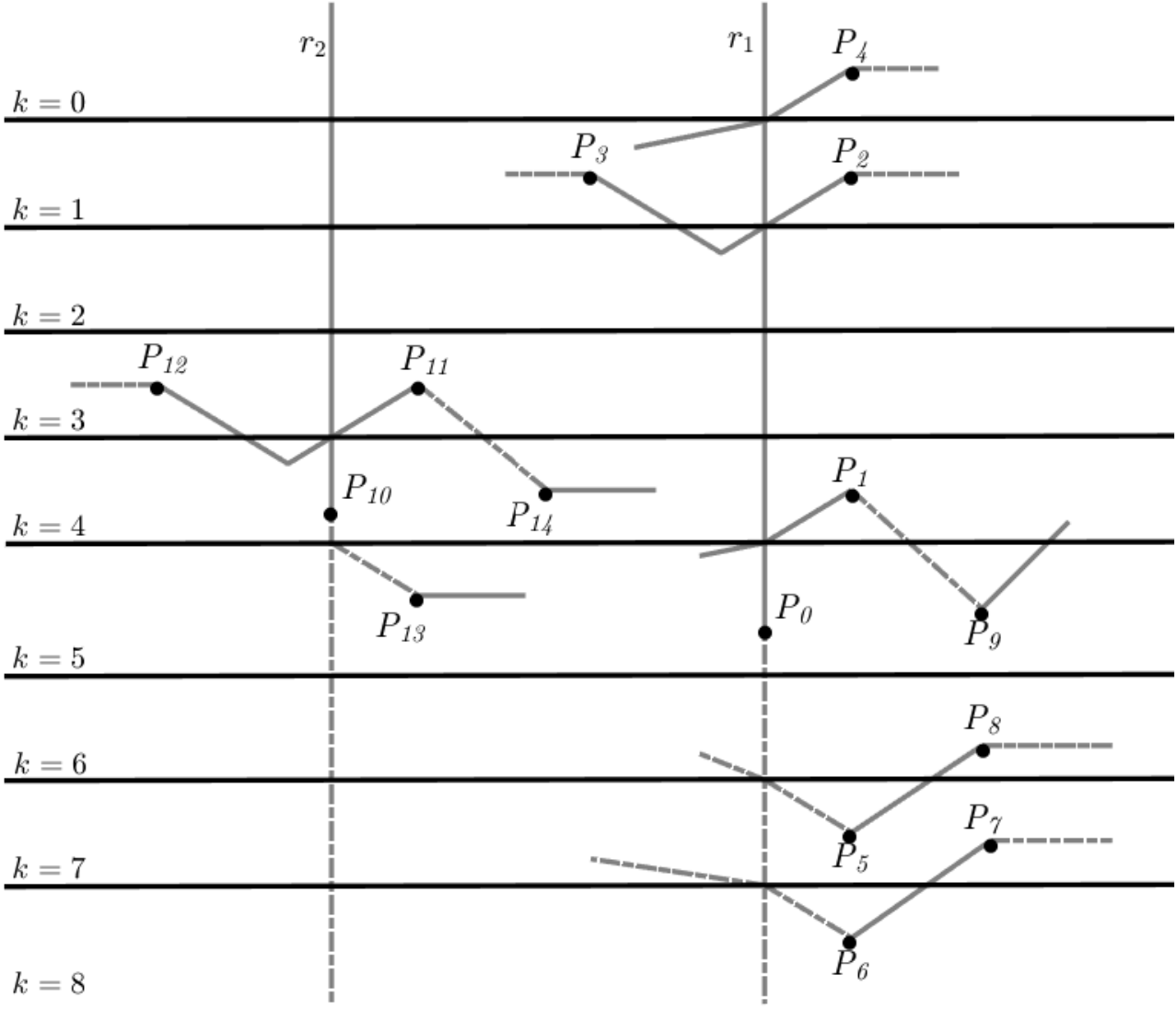}
	\caption{Two $\mathcal{B}$'s for $k=8$}
	\label{fig:twolines}
\end{minipage}
\hfill
\end{figure}

The geometry relates naturally with standard properties: $k$ is the number of negative eigenvalues of $A^h - D^O$, the Morse index in this context (the problem admits a variational formulation). The location of the solutions $P_i$ in terms of the Morse index is in accordance with the continuous case \cite{COSTA}. The parameter $k$ also relates to oscillation theory \cite{CODDINGTON} in the discretized context \cite{GANTMACHER}, an approach which does not extend to, say, the higher dimensional Dirichlet Laplacian. 

The fact that for small $k$ one expects $2k$ solutions should be compared with the example in Figure \ref{figbabado}. In that case, $k$ critical components give rise to  $k + 1$ solutions and the critical set contains only folds. Additional solutions indicate critical points which are not folds, in the spirit of Section \ref{winding}.

We now take  parameters $\ell_-=0.4984$ e $\ell_+=56.9367$, for which  $k=8$ and use the same $g$ (for the finer mesh), with 24 preimages from the table in Section \ref{table}.

We initialize the algorithm differently: choose orthants randomly, and search for a solution by solving the associated linear system. A first solution $P_0$ (notice that this is not the solution given in equation (\ref{LazerMcKenna})), with $k= 5$, was obtained after drawing 127 orthants. Define the generic line
\begin{align}
r_1=&\{ P_0+ s(0.8\sin(I_h)-0.1\sin(2I_h)-0.1\sin(3I_h)-0.1\sin(4I_h)\nonumber \\ &-0.1\sin(5I_h)-0.1\sin(6I_h)-0.1\sin(7I_h)+0.1\sin(8I_h))\}\nonumber
\end{align}
containing $P_0$.
The computation of  the bifurcation diagram $\mathcal{B}$ led to 17 new candidate preimages, of which only 9 had a small relative error, of the order of $10^{-15}$. The spurious solutions are related to inversion
of unstable matrices of the form $(A^h-D^o)$ in orthants trespassed along the homotopy process. Stretches of the bifurcation diagram are either continuous or dotted, depending if they were obtained from $s=0$ by taking $s$ positive or negative. Points $P_i$ at which continuous and dotted stretches intercept are indeed regular points of $F$.

A new solution $P_{10}$ is obtained by sampling additional orthants, as in the computation of $P_0$. The bifurcation diagram associated with $r_2=\{P_{10}+s(\sin(I_h)-\sin(8I_h))\}$ yielded four more preimages, after filtering candidates by relative error. Figure \ref{fig:twolines} shows  the (truncated) bifurcation diagrams associated with  $r_1$ and $r_2$.


Additional solutions were obtained by sampling, leading to $24$ solutions. The table below counts solutions $u$ by $k$, the number of negative eigenvalues of $DF(u)$.

\begin{table} [h]
\centering
\begin{tabular}{|c|c|c|c|c|c|c|c|c|c|}
	\hline
	$k$ & 0 & 1 & 2 & 3 & 4 & 5 & 6 & 7 & 8 \\ \hline
	$\hbox{Number of preimages}$ & 1 & 2 & 2 & 4 & 6 & 4 & 2 & 2 & 1 \\ \hline
\end{tabular}
\end{table}

%


\section{Semi-linear perturbations of the Laplacian} \label{Solimini}

We present a version of the Ambrosetti-Prodi theorem, combining material in \cite{AP,MM, BERGER, CALNETO, SMILEY}).  On a bounded set $\Omega \subset \mathbb{R}^n$ with smooth boundary, let  
\[ -\Delta_D: X = H^2(\Omega) \cap H^1_D(\Omega) \to Y=L^2(\Omega)\]
 be the Dirichlet Laplacian, and denote its smallest eigenvalues by $\lambda_1 < \lambda_2$. A {\it vertical line} is a line in $X$ or $Y $ with direction given by $\phi_1>0$, a positive eigenvector associated with $\lambda_1$. {\it Horizontal subspaces} $H_X \subset X $ and $H_Y\subset Y$  consist of vectors perpendicular to $\phi_1$. {\it Horizontal  hyperplanes}  are parallel to the horizontal subspaces.

\medskip
\begin{theo} \label{APT} 	Consider the function
\begin{equation} \label{AP}
	F: X \to Y \ , \quad u \mapsto  - \Delta_D u - f(u) 
\end{equation}
where $f: \RR \to \RR$ is a  strictly convex smooth function satisfying
\begin{equation} \label{ASY} - \infty < \lim_{x \to - \infty} f'(x) < \lambda_1 < \lim_{x \to  \infty} f'(x) < \lambda_2 \ .
\end{equation}
The critical set $\cC$ of $F$ contains only folds.
The orthogonal projection $\cC \to H_X$ is a diffeomorphism, as is the projection of the image of each horizontal hyperplane $F(H_V + x) \to H_Y, x \in X$. 
The inverse under $F$ of vertical lines in $Y$ are curves in $X$  intercepting each horizontal hyperplane and $\cC$ exactly once, transversally. In particular, $F$ is a global fold and the equation $F(u) = g \in Y$ has 0, 1 or 2 solutions.
\end{theo}

\medskip
In the spirit of Section \ref{basicfolds}, $F: X \to Y$ is a {\it global fold} if
there are diffeomorphisms $\Phi: \RR \times Z \to X $ and  $\Psi: Y \to \RR \times Z$ such that
\[ \tilde F = \Psi \circ F \circ \Phi(t, z): \RR \times Z \to \RR \times Z  \ , \quad \tilde F (t, z) =  (t^2, z)\] for some real Banach space $Z$. The statement implies that the flower of $F$ equals the critical set $\cC$: compared to the examples in the previous sections, the global geometry of $F$ is very simple. Both functions $F$ and $\tilde F$ trivially satisfy the geometric model in the Introduction: domain and counterdomain split in two tiles, both topological half-spaces, and both tiles in the domain are sent to the same tile in the counterdomain. Said differently, the flower equals the critical set, which is topologically a hyperplane.

The underlying geometry led to numerical approaches to solving $F(u) = g$.
Smiley (\cite{SMILEY}) suggested an algorithm based upon one-dimensional searches, later implemented in \cite{CALNETO}. Finite dimensional reduction applies for (generic) asymptotically linear functions $f$  for which the image of $f'$ contains a finite number of eigenvalues of $-\Delta_D$ (\cite{CALNETO,KAMINSKI}).
Under hypothesis (\ref{ASY}), the {\it nonconvexity} of $f$ implies that some right hand side $g$ has four preimages (\cite{CaTZ}). Up to technicalities, the algorithm yields all solutions for convex and nonconvex nonlinearities.

\medskip
Following a different geometric inspiration, Breuer, McKenna and Plum (\cite{PLUM}) computed four solutions  of
\begin{equation} \label{Plum} - \Delta u + u^2 \ = \  800 \sin( \pi x) \sin(\pi y), \ (x,y ) \in \Omega = (0,1) \times (0,1) \ , \ u|_\Omega = 0 \ . 
\end{equation}
The hardest one, call it $u$, was obtained 
by interpreting it as a saddle point  of a functional associated with the variational formulation of the equation. The authors present a computer assisted proof that $u$ is reachable by Newton's method from a computed initial condition $\tilde u$. We do not operate on this level of detail. Such four solutions were also obtained in \cite{ALLGOWER2}. 

Equation (\ref{Plum}) may be treated with our methods, but we introduce a situation with additional difficulties

In preparation, we count intersections of  vertical lines and $\cC$. For a bounded, smooth domain $\Omega \subset \RR^n$, the spectrum of  $- \Delta_D: X \to Y$ is
\[ 0 < \lambda_1 \le \lambda_2 \le \ldots \le \lambda_k <  \lambda_{k+1} \le \ldots \to \infty \ , \]
with associated (normalized) eigenvectors $\phi_k$, where $\phi_1 >0$ in $\Omega$.
Let $f: \RR \to \RR$ be a smooth function for which $\lim_{x \to \pm \infty} f'(x) = \ell_{\pm}$, where $\lambda_- < \lambda_1$ and $\lambda_k \in (\lambda_k, \lambda_{k+1})$. To simplify some arguments, we also consider 
\[ \tilde F: \tilde X \to \tilde Y \ , \quad \tilde F(u) = -\Delta u - f(u)\]
where $\tilde X = C^{2,\alpha}_D(\Omega)$, the  H\"older space of functions equal to zero on $\partial \Omega$, and $\tilde Y =  C^{0,\alpha}(\Omega)$, $\alpha \in (0,1)$. For $w \in \tilde X$, consider the {\it vertical line}  $\{w + t \phi_1,  t \in \RR\}$.

\medskip
We use a natural extension of the familiar Rayleigh-Ritz technique to semibounded operators, Lemma XIII.3 of \cite{RS}, which we transcribe without proof.

\medskip
\begin{prop} \label{RayleighRitz}
Let $H$ be a complex Hilbert space, $T: D \subset H \to H$ a  self-adjoint operator bounded from below with bottom eigenvalues
$\mu_1 \le \mu_2 \le \ldots \le \mu_k$ counted with multiplicity. Let $V$ be an $n$-dimensional subspace $V \subset D$, $n \ge k$. Let $P$ be the orthogonal projection $P: H \to V$ and suppose that the restriction $P T P^\ast: V \to V$ has   eigenvalues $\nu_1\le \ldots \nu_n$. Then $\mu_i \le \nu_i, i=1, \ldots, k$.
\end{prop}

\medskip
\medskip
\begin{prop} \label{Minmax} Vertical lines intercept the critical set $\cC$ of $\tilde F$ at at least $k$ points (counted with multiplicity).
\end{prop}

\begin{proof}
For the smooth function $\tilde F$, the Jacobian $D\tilde F(u): \tilde X \to \tilde Y$ is given by $D\tilde F(u) v = - \Delta v - f'(u) v$. From standard arguments in linear elliptic theory, since $q = f' \circ u$ is a bounded (continuous) function, the operator 
\[ T(u): X \to Y , \quad v \mapsto - \Delta_D v - f'(u) v\ \]
is self-adjoint with spectrum consisting of eigenvalues
\[ \mu_1(u) < \mu_2(u) \le \mu_3(u) \le \ldots  \to \infty \ . \]
Moreover, $\sigma(D\tilde F(u)) = \sigma(T(u))$.
We first show that for  $t <<0$, $T(t)$ is a positive operator (i.e., all its eigenvalues are strictly positive). Indeed, by dominated convergence, since $w + t \phi_1 \to \ell_-$ pointwise as $t \to -\infty$,
\[ \mu_1(u(t)) = \min_{\| v \| = 1} \langle T(u(t)) v, v \rangle = \min_{\| v \| = 1} \langle -\Delta_D v , v \rangle  - \langle f'(w + t \phi_1) v, v \rangle
\ge \lambda_1 - \ell_- > 0 \ .\]
We now obtain estimates for $\sigma(T(u(t))$ for $t >> 0$. Let $V = \spam\{\phi_i, i=1, \ldots,k\}$, the vector space generated by the first $k$ eigenfunctions of $-\Delta_D$. Then the matrix $M(t)$ associated with $P H P^T$ in this (orthonormal) basis for $V$ has entries
$M_{ij} = M_{ji} = \delta_{ij} \lambda_i  - \langle v_i, f'(w + t \phi) v_j \rangle$. Again, as $t \to \infty$, $\langle v_i, f'(w + t \phi) v_j \to \delta_{ij} \ell_+$, so that $M(t)$ converge to the diagonal matrix with diagonal entries $\lambda_i - \ell_+ < 0$. From Proposition \ref{RayleighRitz}, the first $k$ eigenvalues of $T(u(t))$  are strictly negative for large $t$.
\end{proof}

\begin{figure} 
	\begin{centering}
		\includegraphics[height=150pt,width=150pt]{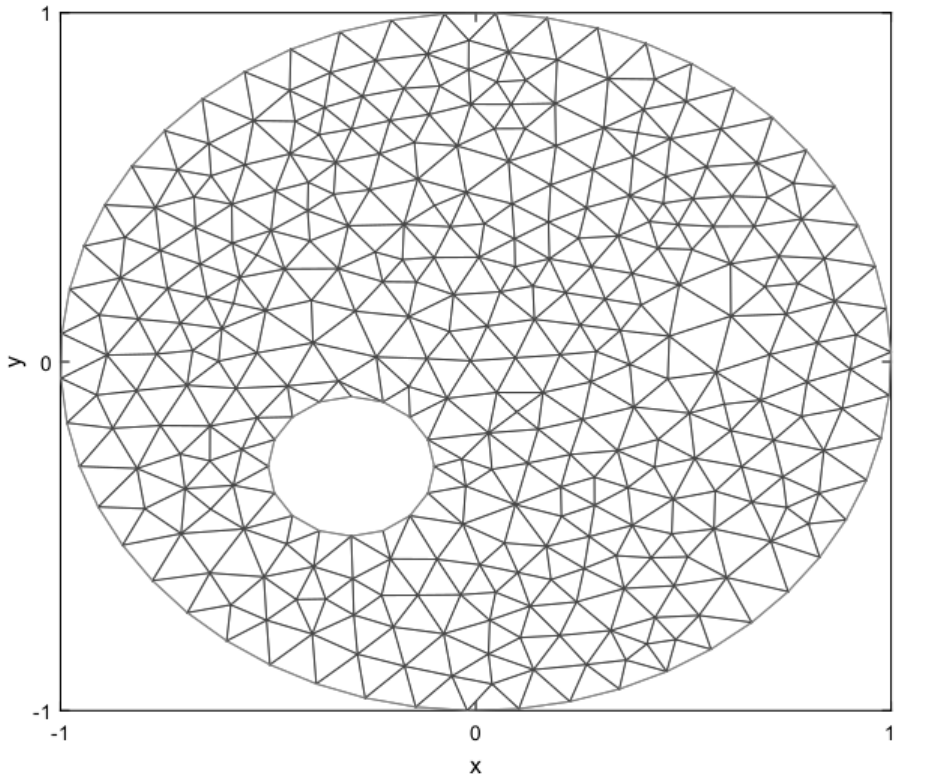}
		\caption{The domain $\Omega$ and its mesh}
		\label{fig:5.1}
	\end{centering}
\end{figure}

\bigskip
In \cite{LMCKENNA}, Lazer and McKenna conjectured that, for an asymptotically linear $f$ with parameters $\ell_-<\lambda_1$ and $\lambda_k<\ell_+<\lambda_{k+1}, k \in \mathbb{N}$, there should be at least  $2k$ solutions of (\ref{eq:PVCnaolinear}) for  $g = - t \phi_1, t >> 0$. A counterexample was provided by Dancer
(\cite{LAZERANDMCKENNA}). The example we handle follows a positive result of Solimini \cite{SOLIMINI}. 

Let  $\Omega$ be the annulus \[ \Omega=\{x \in  \mathbb{R}^2 | \ |x|<1 \ ,  \ |x-(-0.3,-0.3)|>0.2 \} \ . \] 

We discretize  $- \Delta_D$
by piecewise linear finite elements on a mesh on $\Omega$ with 274 triangles.
The four smallest eigenvalues of the discretized operator $-\Delta_D^h$ are simple,
\[ \lambda_1^h \approx  9.0988, \quad \lambda_2^h \approx 16.3218, \quad \lambda_3^h \approx 22.9346, \quad \lambda_4^h \approx 30.4949 \ . \]

We consider
\begin{equation} 	\label{eq:PVCnaolinear}
	F(u)=-\Delta_D u-f(u)=g, \ \ \  \ u|_{\partial \Omega}=0 .
\end{equation}

According to Solimini, for some $\epsilon>0$ and parameters $\ell_-$ and $ \ell_+$ satisfying
$ \ell_-<\lambda_1 < \lambda_3< \ell_+< \lambda_3+\epsilon$
the equation (\ref{eq:PVCnaolinear}) has {\it exactly} $6$ solutions for $g = - t \phi_1$ for  large, positive $t$.
For concreteness, we take $f$ such that $f'(x)=\alpha \ \text{arctan}(x)+\beta$, where $\alpha$ and $\beta$ are adjusted so that $\ell_- = -1$, $\ell_+ = 25.3397$. Finally, set $g^h=-1000 \ \phi_1^h$.

\bigskip
We first obtain a solution $P_0$ by a continuation method.
Set
$ r=\{P_0+s(0.8\phi_1^h - 0.1\phi_2^h - 0.1\phi_3^h)$, a stretch of which ($s \in [-1000, 1000]$) we traverse with an increment $h_s=0.1$. As in the previous section, small terms are inserted so as to increase the possibilities that intersections with the critical set are transversal.

Along $r$, the four smallest eigenvalues $\mu_i^h$ of the Jacobian $DF(u)$ are given in Figure \ref{fig:5.2} (for the underlying numerics, we used \cite{BOFFI}). A point $u \in r$ is a critical point of $F$ if and only if some such eigenvalue is zero.

\begin{figure} 
\centering
\begin{minipage}{0.45\textwidth}
	\centering
	\includegraphics[width=165pt]{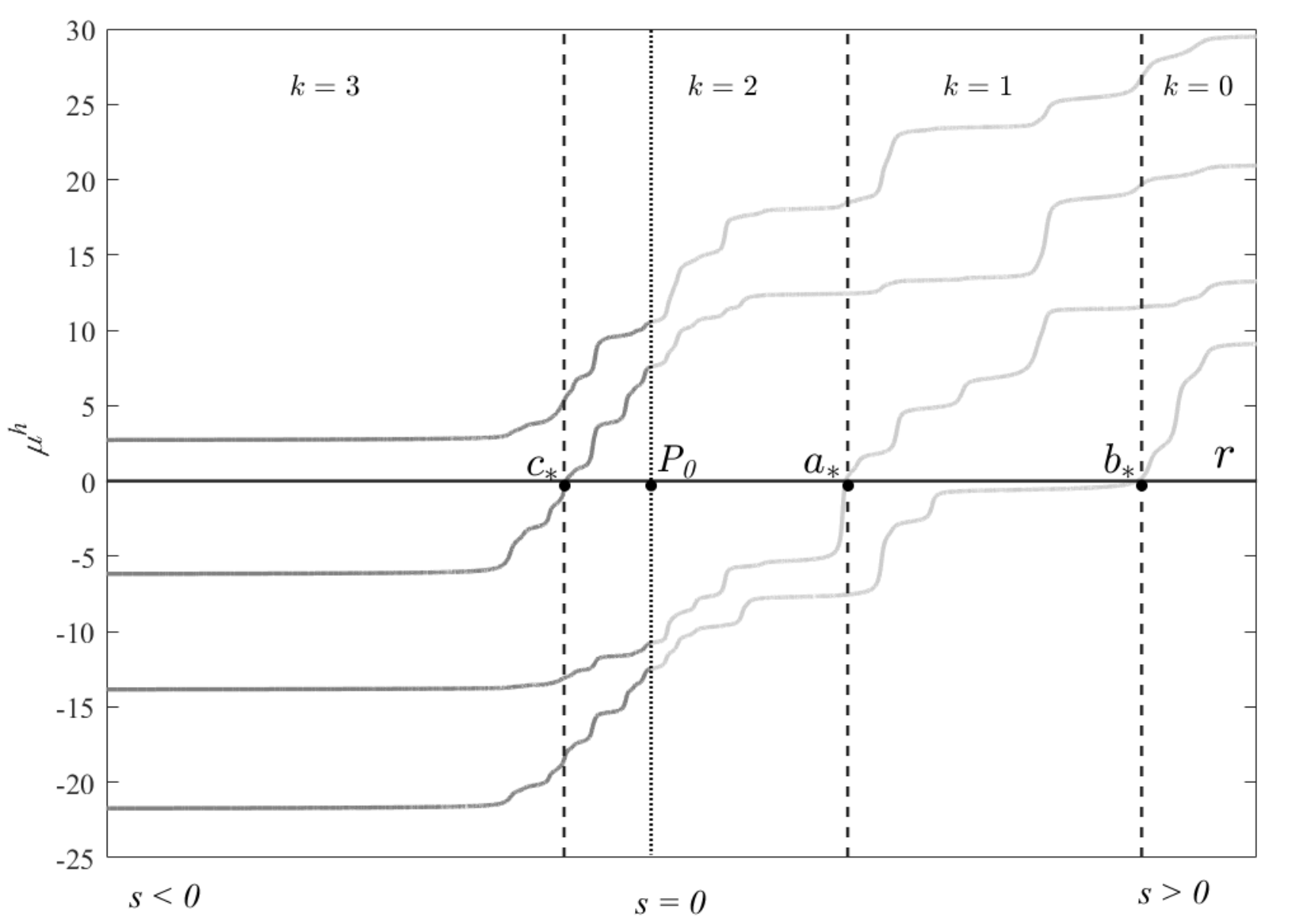}
	\caption{The four smallest eigenvalues of $DF(u)$ for $u(s) \in r$.}
	\label{fig:5.2}
\end{minipage}
\hspace{0.4cm}
\begin{minipage}{0.45\textwidth}
	\centering
	\includegraphics[width=150pt]{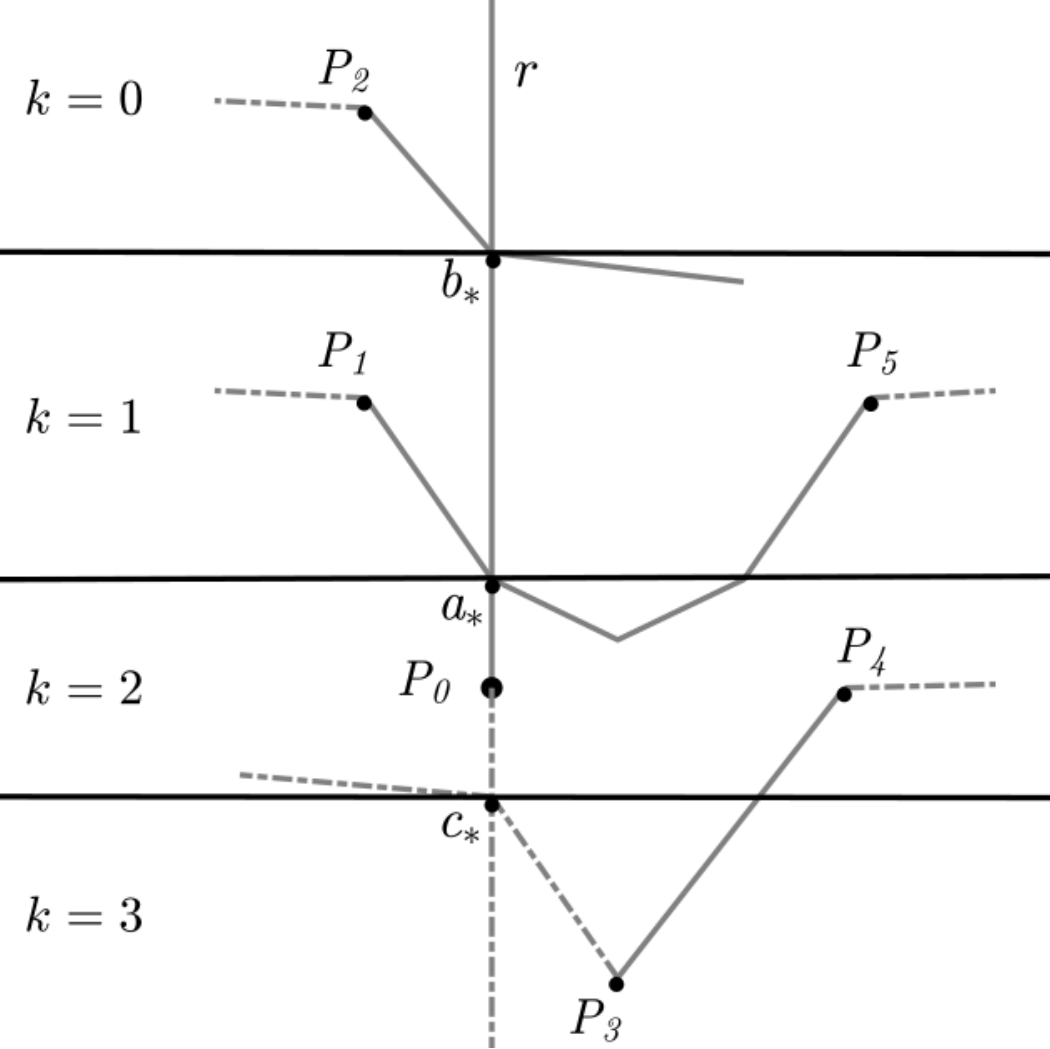}
	\caption{The six solutions in $\mathcal{B}$}
	\label{fig:morse}
\end{minipage}
\end{figure}

In Figure \ref{fig:morse}, horizontal lines represent parts of the critical set $\mathcal{C}$. The value $k$  counts the number of negative eigenvalues of $DF(u)$ at each of the regular components. The point $P_0 \in r$ belongs to a component for which $k=2$. The bifurcation diagram $\mathcal{B}$, containing the six solutions, is described in Figure \ref{fig:morse} and the solutions are given in Figure \ref{fig:5.5}. Continuation to $P_5$ required  finer
jumps along $r$.

\begin{figure} 
$$\begin{array}{cc}	
	\includegraphics[keepaspectratio,scale=0.13]{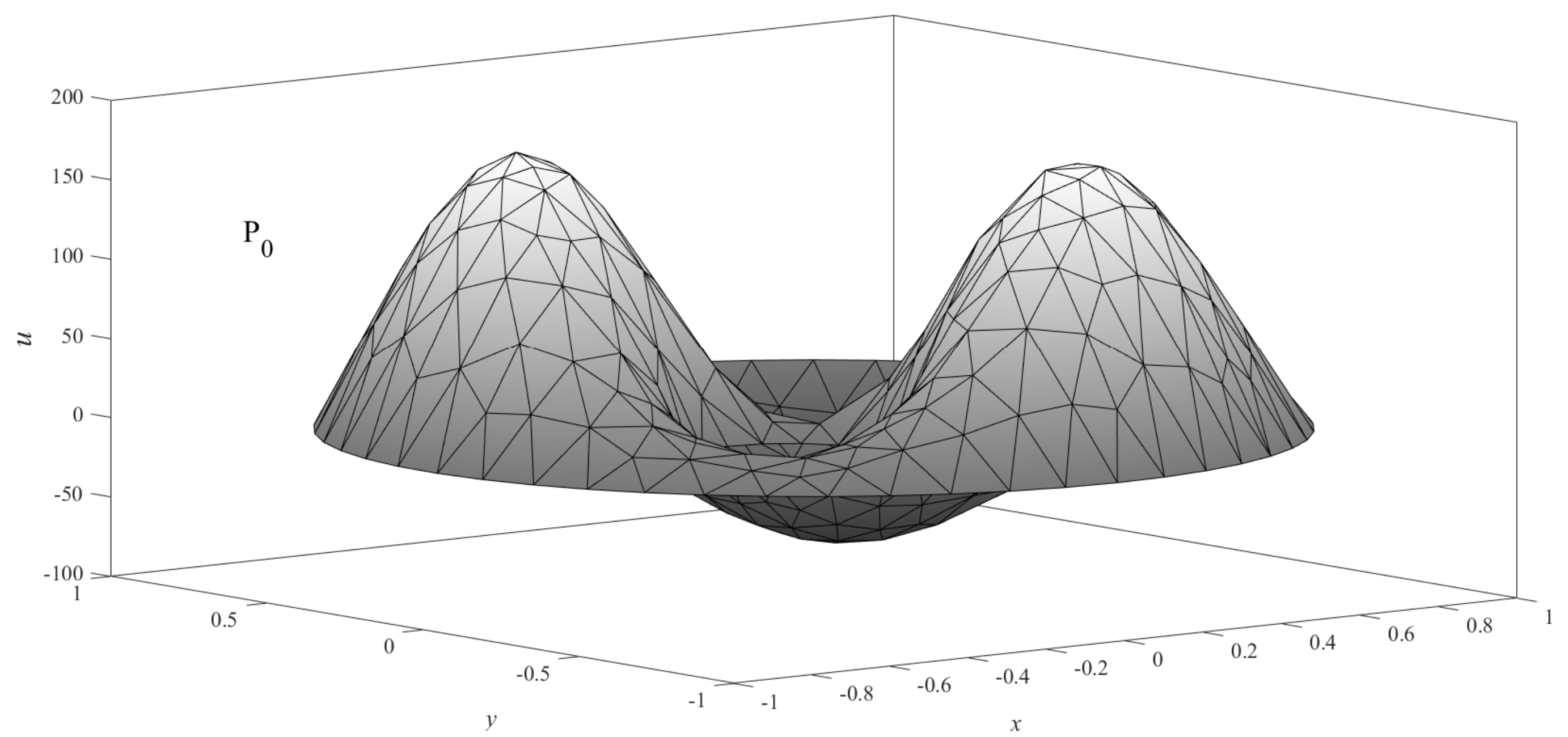} &
	\includegraphics[keepaspectratio,scale=0.13]{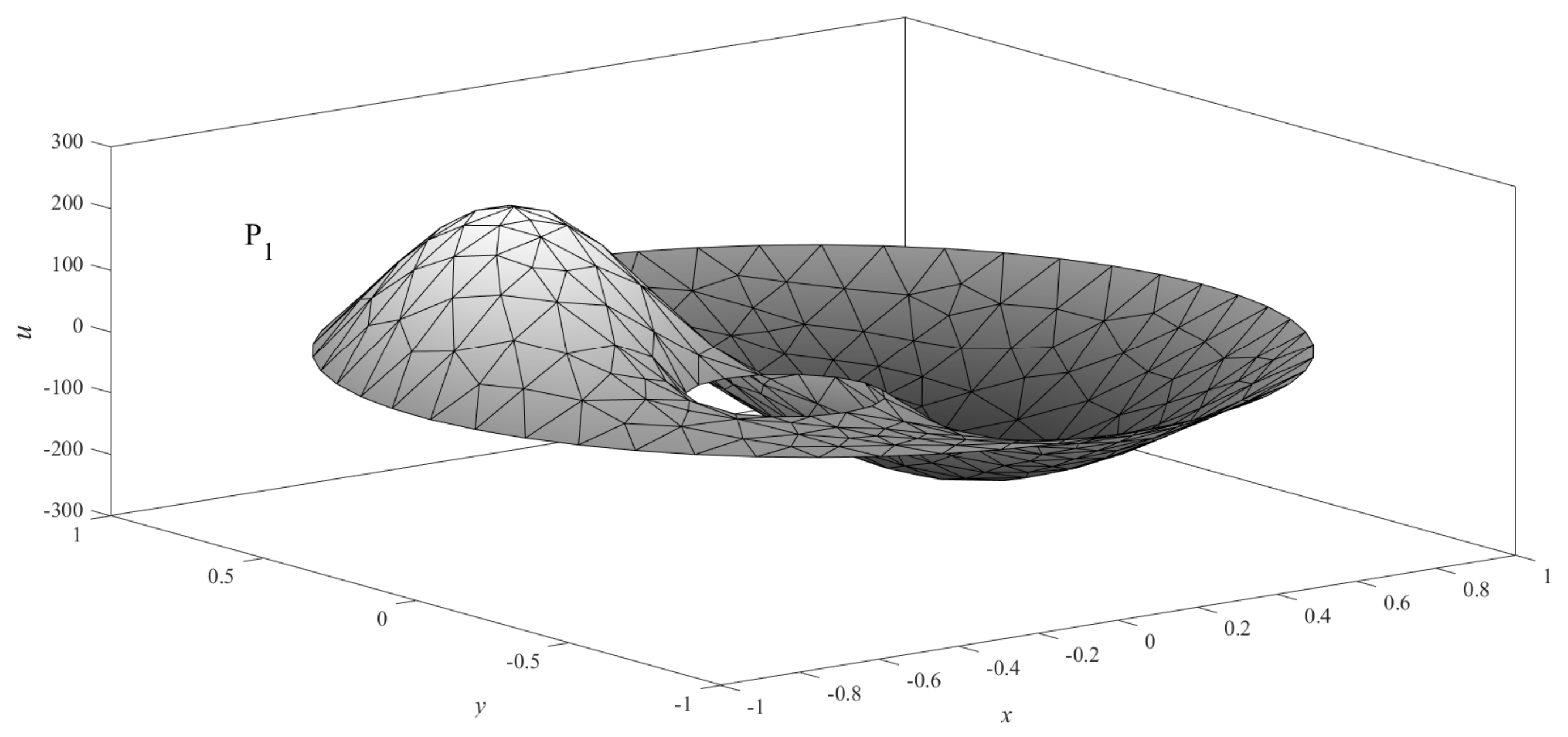} \\
	\includegraphics[keepaspectratio,scale=0.13]{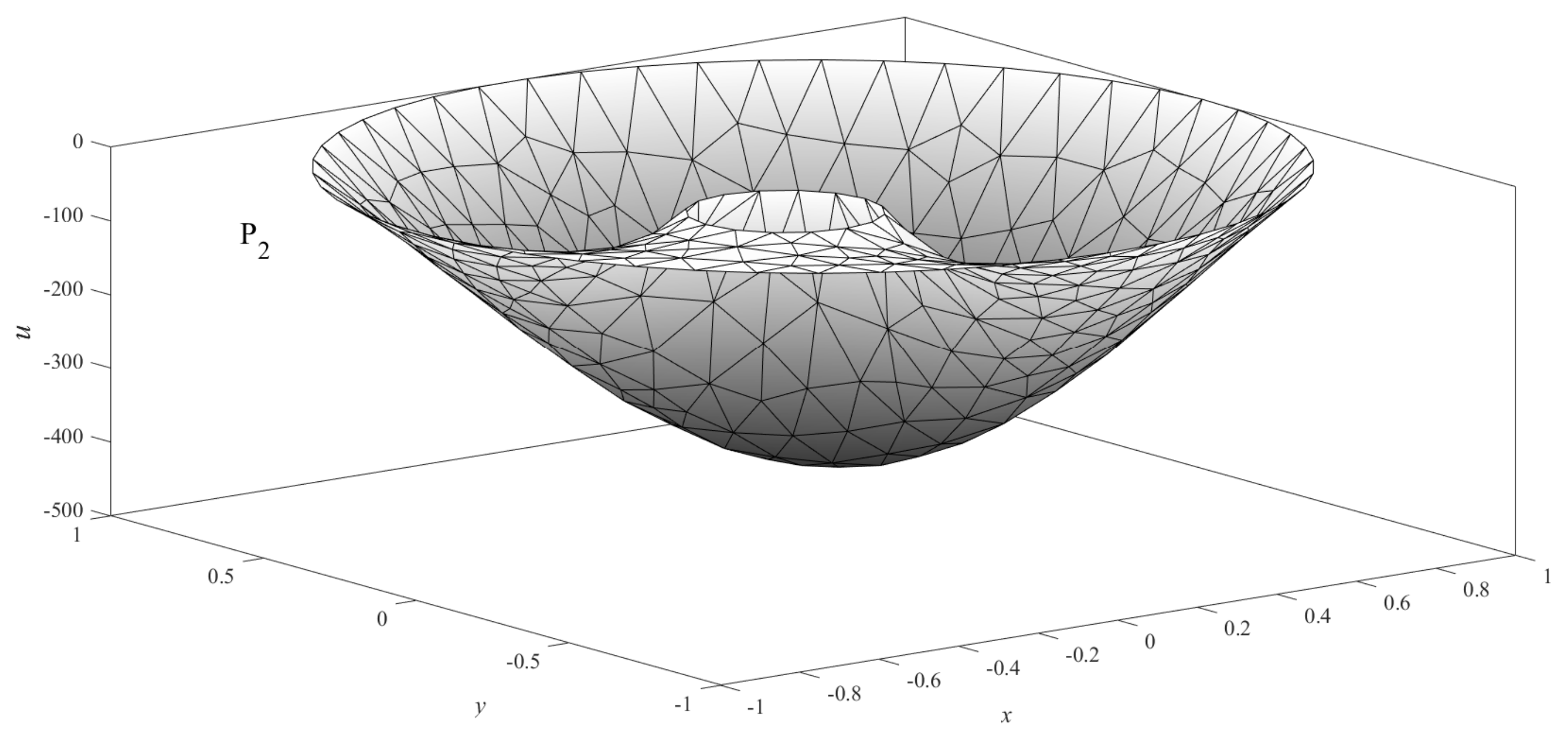} &
	\includegraphics[keepaspectratio,scale=0.13]{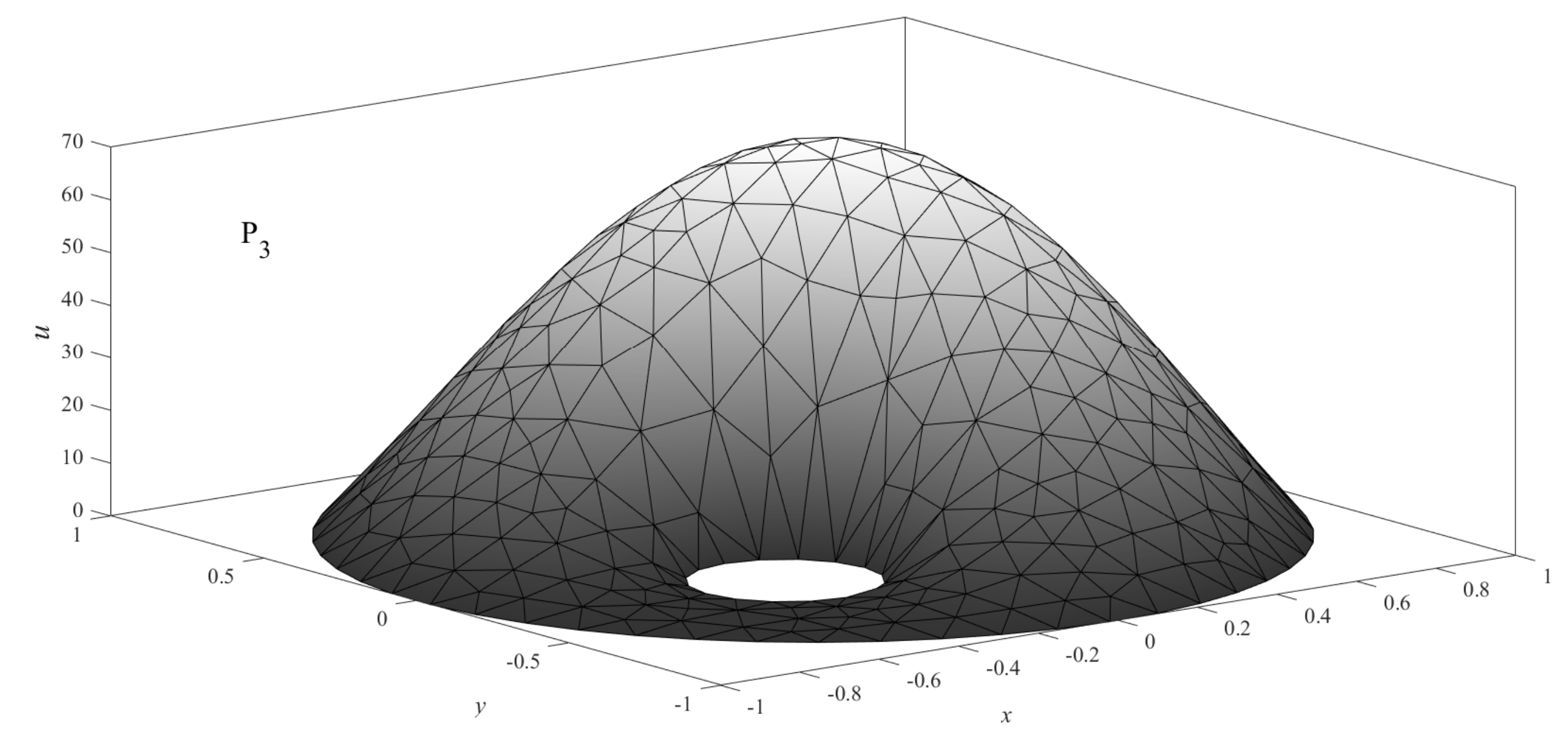} \\
	\includegraphics[keepaspectratio,scale=0.13]{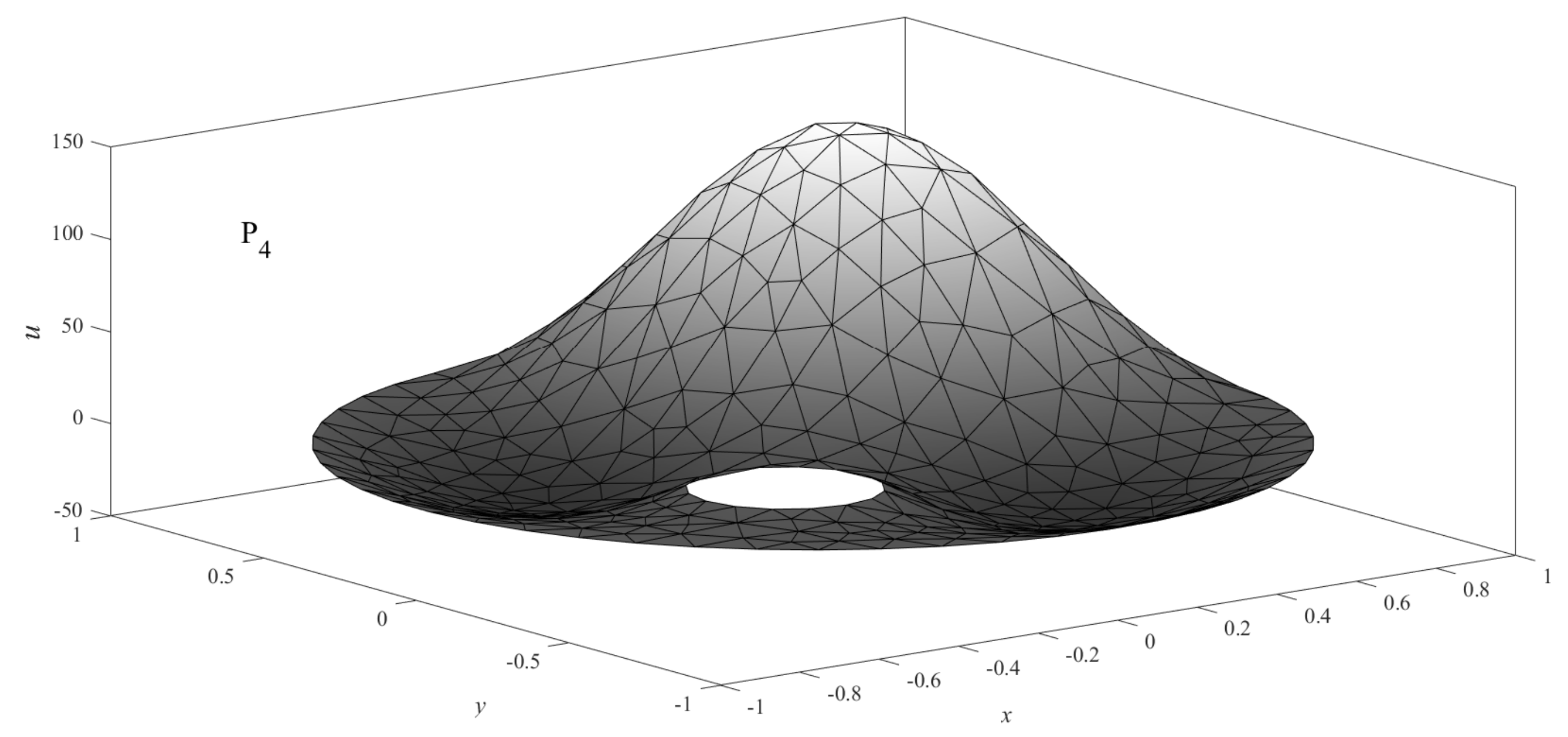} &
	\includegraphics[keepaspectratio,scale=0.13]{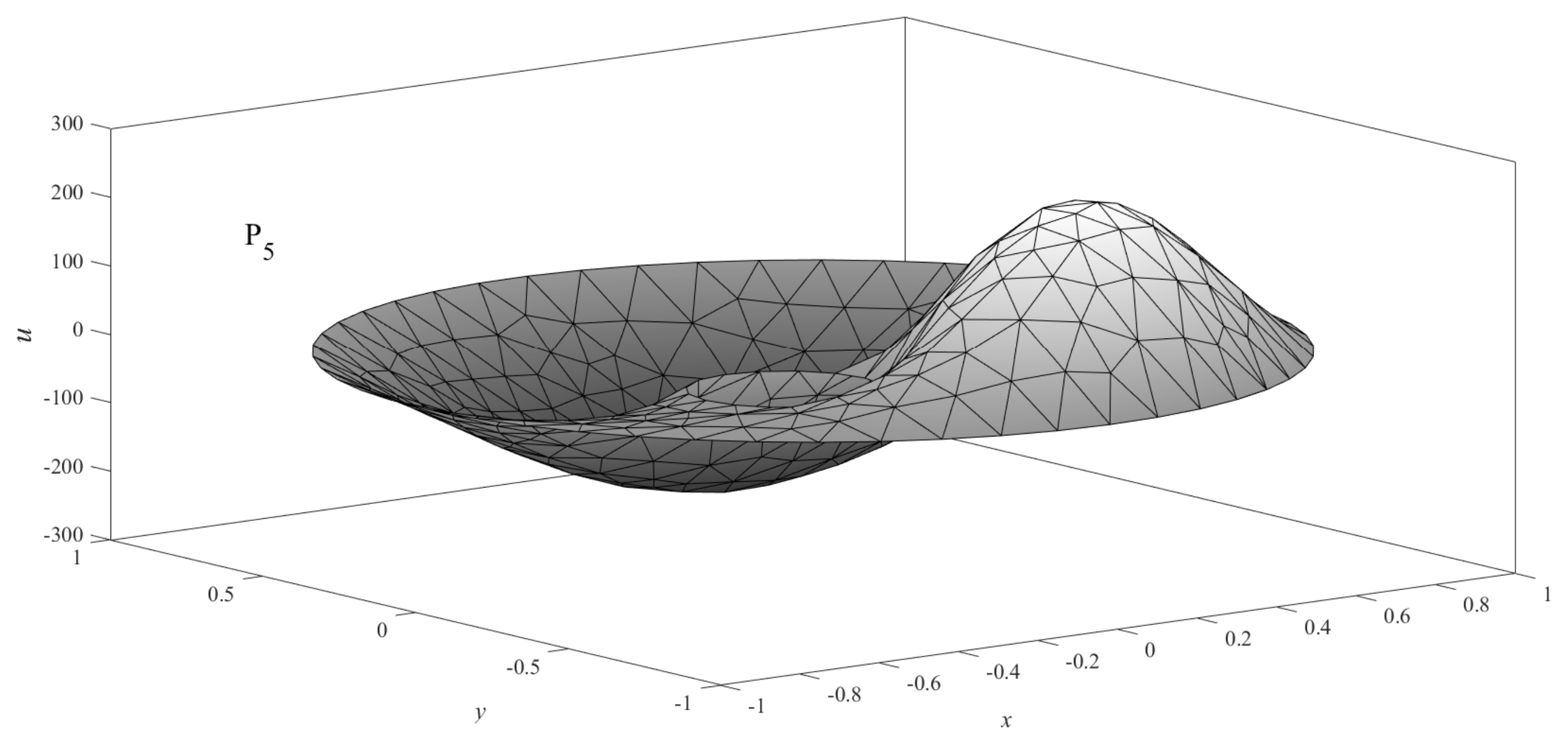}
\end{array}$$
\caption{The six solutions}
\label{fig:5.5}
\end{figure}

No relative residue
$\epsilon({u^h})=\frac{||F^h(u^h)-g^h||_{Y^h}}{||g^h||_{Y^h}}$ is larger than $10^{-12}$.

%
%
%

%

\bigskip
\noindent{\bf Appendix - continuation at a fold from spectral data} \label{continuationatafold}

\bigskip
Predictor-corrector methods at regular points are well described in the literature (\cite{ALLGOWER,KELLER,KELLEY,RHEINBOLDT}). Here we provide details about the inversion algorithm we employ in the examples in Section  \ref{Solimini} in the neighborhood of a fold.

The algorithm must identify critical points $u$ of $F:X \to Y$. Due to the nature of the examples, this is accomplished by checking if some eigenvalue of the Jacobian $DF(u)$ is zero. We assume that the original problem $F(u) = g$ admits a variational formulation, so that $DF(u)$ is a self-adjoint map, and the task is simpler. The general case may also be handled, but we give no details.

We modify the prediction phase of the usual continuation method and perform correction in a standard fashion.
Following (\cite{UECKER, CALNETO,KAMINSKI}), we use spectral data: by continuity, for $u$ close to a fold $u_c$, the Jacobian $DF(u)$ has an eigenvalue $\lambda$ close to a zero eigenvalue $\lambda_c=0$ of $DF(u_c)$, and a normalized eigenvector $\phi$ close to $\phi_c$, a normalized generator of $\ker DF(u_c)$. The eigenvalue $\lambda$  plays the role of arc length in familiar algorithms.

For a smooth function $F: X \to Y$ between real Banach spaces, we search for the preimage $u(t)$ of a smooth curve $ \gamma(t) \subset Y$  such that, at $t = t_c$, $\gamma(t_c) = F(u_c)$ is the image of a fold $u_c$. As usual, we consider the homotopy
$$
\begin{array}{lrll}
	H:&X \times \mathbb{R} &\longrightarrow Y, \ \ (u,t)&\longmapsto F(u)- \gamma(t) \\
\end{array}
$$
and assume the hypothesis of the implicit function theorem:
$$
\begin{array}{lrll}
	DH(u,t): X \times  \mathbb{R} \to  Y \ , \quad (\hat u, \hat t) \ \mapsto DF(u) \ \hat u - \gamma' (t) \ \hat t
\end{array}
$$
is surjective at $(u_c, t_c)$. Clearly, $\gamma'(t) \in Y$.

\medskip
\begin{prop}  $DH(u_c, t_c)$ is  surjective if and only if $\gamma'(t_c) \notin \Ran DF(u_c)$.
\end{prop}

\medskip
Geometrically, the curve $\gamma(t) \in Y$ crosses the image of the critical set $F(\mathcal{C})$ transversally at the point $\gamma(t_c) = F(u_c)$. Since  $\gamma$ is chosen by the programmer, this is no real restriction.

\begin{proof}
	As $u_c$ is a fold, $DF(u_c)$ is  a Fredholm operator of index 0 with one dimensional kernel, and  image given by a closed subspace of codimension one. Surjectivity of $DH(u_c, t_c)(\hat u , \hat t) = DF(u_c) \hat u - \gamma'(t_c) \hat t$ holds exactly if $\gamma'(t)$ generates a complementary subspace to $\Ran DF(u_c)$.
\end{proof}

The next proposition ensures that the inversion of appropriate operators may be performed as in the finite dimensional case.  If $ X = Y = \RR^n$, $DH(z)$ is an $n \times (n+1)$ matrix of rank $n$.

\medskip
\begin{prop}
	\label{prop:JacobianaFredholm} For $(u,t)$ close to $(u_c, t_c)$, the Jacobian
	$DH(u,t): X \times \RR \to Y$ is a Fredholm operator of index 1. If $\gamma'(t_c) \ne 0$, $\dim \ker DH(u,t)= 1$ if and only if $u = u_c$, otherwise it is zero.
\end{prop}

\begin{proof} To show that $DH(z_c)$ is a Fredholm operator of index 1 with one dimensional kernel,
	set $\RR \sim \{  \gamma'(t_c) \hat t , \hat t \in \RR\}$ and write $DH(z_c)$ as the composition
	\[( \hat u , \hat t) \in X \times \RR \mapsto (DF(u_c) \hat u , \gamma'(t_c) \hat t) \in Y \times \RR \ni  (y, s)   \mapsto y - s \in Y \ , \]
	easily seen to consist of Fredholm operators of indices $0$ and $1$ respectively. Recall that the composition of Fredholm operators yields another Fredholm operator and indices add. We then have that $DH(u_c, t_c)$ is a Fredholm operator of index 1. If $\gamma'(t_c) \ne 0$, then $\dim \ker DH(u_c, t_c) \le 1$ and $\dim \ker DF(u_c) = 1$ if and only if $\dim \ker DH(u_c, t_c)=1$. Let $L$ be Fredholm. By standard perturbation properties of Fredholm operators,  $L+P$ is also Fredholm, $\ind (L+P) = \ind L$ and $DH(u,t)$ is Fredholm of index 1, for $(u,t)$ near $(u_c, t_c)$, as $H$ is of class $C^1$. Smoothness of  eigenvalues and eigenvectors proves the claims for $\ker DH(u,t)$.
\end{proof}

To obtain a prediction from point $(u,t) = (u(t),t) \in H^{-1}(0)$, we must find a nonzero tangent vector $(\hat u, \hat t) \in T_{(u,t)} H^{-1}(0)$, so that
$$
\begin{array}{lrll}
	DH(u,t)(\hat u, \hat t)=DF(u) \hat u - \gamma'(t) \hat t =0 \ , \quad (\hat u, \hat t) \ne 0 .
\end{array}
$$

At points $u= u(t)$ for which $DF(u)$ is invertible, this is easy: set $\hat t = 1$ and get $\hat u $ by solving a linear system. Instead, we assume $u$ close to a fold $u_c \in \mathcal{C}$. By the smoothness of simple eigenvalues and associated (normalized) eigenvectors,  $DF(u)$ has an eigenvalue $\lambda$ and associated eigenvector $\phi$  near an eigenvalue  $\lambda_c = 0$ and eigenvector  $\phi_c$ of $DF(u_c)$. We must compute a nonzero solution $(\hat u ,\hat t)$ of $DH(u,t)(\hat u ,\hat t) = 0$, or equivalently $DF(u)\hat u = - \gamma'(t) \hat t$,  with a procedure which is continuous in $t \sim t_c$.

For $\hat u \in X$, split $ \hat u= \hat v + \hat r \phi$ for $\langle \hat v,  \phi \rangle =0$ and $\langle \phi, \phi \rangle =1$. Clearly, $\hat v$ and $\hat r$ are continuous in $u$, since $\phi= \phi(u)$ is. The tensor product $\phi \otimes \phi$ denotes the rank one linear map $(\phi \otimes \phi) v = \langle \phi, v \rangle \phi$. In particular, as $\| \phi \|=1$, $(\phi \otimes \phi) \phi = \phi$.

\medskip
\begin{prop}
	Let $u$ be sufficiently close to a fold $u_c \in X$  of $F$, with the eigenvalue $\lambda \sim 0$ such that $| \lambda|< 1$ and associated normalized eigenvector $\phi$. For $\alpha \ge 1$ the operator $S(u)=DF(u)+\alpha \ \phi \otimes \phi  : X \to Y$ is invertible.	
\end{prop}

Notice that $S = DF(u)$ when restricted to $\phi^\perp$.
\medskip

\begin{proof} The operator $S(u)$ is a rank one perturbation of $DF(u)$, and thus it is also a Fredholm operator of index zero: invertibility is equivalent to injectivity. For $\hat u= \hat v + \hat r \phi$, with $\hat v \in \phi^\perp$ and $\langle \phi, \phi \rangle =1 $,
	\[ S(u) \hat u= DF(u) (\hat v + \hat r \phi) + \alpha(\phi \otimes \phi) (\hat v + \hat r \phi) = 0 \ , \]
	implies
	\[ S(u) \hat u= DF(u) \ \hat v  +  \hat r \lambda \phi + \alpha \ \hat r \ \phi  = 0 \ . \]
	Since $DF(u)$ is self-adjoint, $\phi^\perp$ is an invariant subspace and both terms are zero, 
	\[ DF(u) \hat v =0 ,\  (\alpha + \lambda) \hat r \ \phi  = 0 \ . \]
	As the restriction $DF(u)$ to $\ker DF(u)^\perp$ is an isomorphism,  $\hat v = 0$, $\hat r = 0 $.
\end{proof}

\begin{prop} Under the hypotheses of the proposition above, the solution of
	\[ S(u) \hat u = - \lambda \gamma'(t)  - \alpha \langle \phi, \gamma'(t) \rangle \phi \]
	is of the form $\hat u = \hat v - \langle \phi, \gamma'(t) \rangle \phi$ for some $\hat v \in \phi^\perp$, $\| \phi \| = 1$. Moreover,
	\[ DF(u) \hat u = - \lambda \gamma'(t) \ \hbox{and} \ (\hat u , \lambda) \in T_{(u,t)} H^{-1}(0) \ . \]
\end{prop}

In other words, $(\hat u , \lambda)$ is the tangent vector required in the prediction phase.

\begin{proof} As $S(u)$ is invertible, $\hat u$ is well defined for the given right hand side. For $\hat u = \hat v + \hat r \phi$ with $\hat v \in \phi^\perp$ and $\langle \phi, \phi \rangle =1 $,
	take the inner product with $\phi$ of
	\[ S(u) \hat u = DF(u) \hat u + \alpha \hat r \phi = - \lambda \gamma'(t)  - \alpha \langle \phi, \gamma'(t) \rangle \phi \ . \]
	to obtain $\hat r = - \langle \phi, \gamma'(t) \rangle$ and then
	$DF(u) \hat u = - \lambda \gamma'(t)$ follows.
\end{proof}

In the application of Section \ref{Solimini} finite element methods applied to the Laplacian yields the usual sparse matrices. The term $\alpha \phi \otimes \phi$ spoils sparseness, and one has to proceed by inversion through standard techniques associated with rank one perturbations. The numerical inversion worked well with $\alpha=1$.

\end{document}